\documentclass[11pt,reqno]{amsart}
\usepackage{fullpage}

\usepackage{amsfonts,amssymb,amsmath}
\usepackage{mathrsfs}
\usepackage{dsfont}
\usepackage{mathtools}
\usepackage{amsthm}
\usepackage{mathtools}
\usepackage{leftindex,tensor}
\usepackage[OT1]{fontenc}
\usepackage{MnSymbol}
\usepackage{manfnt}
\usepackage[all,cmtip]{xy}
\usepackage{braket}

\usepackage{tikz-cd}
\tikzset{%
   symbol/.style={%
        draw=none,
        every to/.append style={%
            edge node={node [sloped, allow upside down, auto=false]{$#1$}}}
    }
}

\usepackage[pagebackref,hyperindex,linktocpage=true]{hyperref}
\hypersetup{
    colorlinks,
    linkcolor={red!50!black},
    citecolor={blue!50!black},
    urlcolor={blue!80!black}
}

\newtheorem{theo}{Theorem}[section]
\newtheorem{theo-def}{Theorem-Definition}[section]
\newtheorem{lem}[theo]{Lemma}
\newtheorem{lem-def}[theo]{Lemma-Definition}
\newtheorem{prop}[theo]{Proposition}
\newtheorem{cor}[theo]{Corollary}
\newtheorem*{claim}{Claim}

\theoremstyle{definition}
\newtheorem{defi}[theo]{Definition}
\newtheorem{ex}[theo]{Example}

\numberwithin{equation}{section}

\def\a{\alpha}

\def\L{\mathcal{L}}

\def\k{\mathds{k}}

\def\L{\mathcal{L}}

\def\O{\mathcal{O}}
\def\A{\mathcal{A}}

\def\NN{\mathbb{N}}
\def\ZZ{\mathbb{Z}}

\def\QQ{\mathbb{Q}}

\def\bfB{\mathbf{B}}
\def\bff{\mathbf{f}}
\def\bfO{\mathbf{O}}
\def\bfU{\mathbf{U}}

\def\ab{\mathrm{ab}}
\def\ad{\mathrm{ad}}

\def\der{\mathrm{der}}
\def\mod{\mathrm{mod}}

\def\pos{\mathrm{pos}}
\def\la{\lambda}
\def\La{\Lambda}

\DeclareMathOperator{\gr}{gr}
\DeclareMathOperator{\Ind}{Ind}

\DeclareMathOperator{\Spec}{Spec}
\DeclareMathOperator{\Supp}{Supp}
\DeclareMathOperator{\Hom}{Hom}
\DeclareMathOperator{\Rep}{Rep}

\newcommand{\epic}{\twoheadrightarrow}
\newcommand{\into}{\hookrightarrow}

\title{Quantized integral models of reductive monoids}

\author{Jingren Chi}
\email{jrenchi@amss.ac.cn}
\address{Morningside Center of Mathematics and State Key Laboratory of Mathematical Sciences, Academy of Mathematics and Systems Science, Chinese
Academy of Sciences, Beijing 100190, China;}

\author{Simon Jacques}
\email{simon.jacques@amss.ac.cn}
\address{State Key Laboratory of Mathematical Sciences, Academy of Mathematics and Systems Science, Chinese
Academy of Sciences, Beijing 100190, China;}

\date{}

\begin{document}
\begin{abstract}
    We construct quantized integral models of reductive monoids and establish their finite-generation properties. We also describe the quantized coordinate rings of the basic affine spaces explicitly in terms of Lusztig's dual canonical bases. The results will be used in the companion paper \cite{CJ2} to establish a classification theorem for split reductive monoids over arbitrary connected base schemes. 
\end{abstract}
\maketitle

\tableofcontents

\section{Introduction}
\subsection{Motivation}
The construction of Chevalley group schemes as integral models of split reductive groups is an important step towards the classification theorem for reductive group schemes in \cite{SGA3-3new}. There have been at least three approaches to this problem: Demazure's approach in [\emph{loc. cit.}, Expos\'e XXV], Kostant's approach in \cite{Kostant-Z} using an integral form of the universal enveloping algebra, and more recently, Lusztig's construction in \cite{LusztigGroupScheme} using his theory of the (dual) canonical bases \cite{Lusztig-IntroQuantumGroup}. Since the canonical bases are ``quantum" in nature, Lusztig actually constructs quantizations of the integral coordinate rings, i.e. certain (non-commutative) $\ZZ[v,v^{-1}]$-algebras that specialize to the integral coordinate rings of reductive group schemes when $v\mapsto1$. \par 
The major goal of this paper is to construct and study certain subalgebras of Lusztig's quantized coordinate rings, which will provide the integral models of reductive monoids in the companion paper \cite{CJ2}. One important property we will establish is that these are finitely generated $\ZZ[v,v^{-1}]$-algebras. Then after specialization, the classical coordinate algebras are finitely generated over the noetherian ring $\ZZ$, hence finitely presented. Along the way, we also give explicit descriptions of the quantized invariant subring under the action of the maximal unipotent subgroup. Besides being of independent interest, these will also provide technical ingredients in the proof of the uniqueness part of the classification theorem in \cite{CJ2}. 

\subsection{Main results}
Let $G$ be a pinned split reductive group scheme over $\ZZ$ with Borel subgroup $B$ and split maximal torus $T$ such that $T\subset B$. Let $X\coloneqq X^*(T)$ and $Y\coloneqq X_*(T)$ be respectively the weight and the coweight lattice. Let $X^+\subset X$ be the submonoid consisting of dominant weights (with respect to $B$). Then the positive roots (determined by $B$) define a partial order $\le$ on $X$.\par 
Let $\dot{\bfU}$ be Lusztig's modified quantized enveloping algebra over $\QQ(v)$ associated to the based root datum of $(G,B,T)$. Let $\dot{\bfB}$ be the canonical basis of $\dot{\bfU}$ and let \[\dot{\bfB}^*\subset\Hom_{\QQ(v)}(\dot{\bfU},\QQ(v))\] 
be the dual canonical basis.\par 
Let $\A\coloneqq\ZZ[v,v^{-1}]$ and let ${}_\A\bfO$ be the free $\A$-submodule of $\Hom_{\QQ(v)}(\dot{\bfU},\QQ(v))$ spanned by $\dot{\bfB}^*$. Then ${}_\A\bfO$ is a (non-commutative) Hopf $\A$-algebra and the specialization ${}_\ZZ\bfO\coloneqq{}_\A\bfO\otimes_\A\ZZ$ (where $v\mapsto 1$) is a commutative Hopf algebra and it is shown in \cite{LusztigGroupScheme} that the affine group scheme $\Spec({}_\ZZ\bfO)$ is isomorphic to $G$. 
The canonical basis has a partition into the so-called ``two-sided cells":
\[\dot{\bfB}=\bigsqcup_{\la\in X^+}\dot{\bfB}[\la]\]
where each $\dot{\bfB}[\la]$ is a finite subset. Let $\dot{\bfB}^*=\bigsqcup\limits_{\la\in X^+}\dot{\bfB}[\la]^*$ be the corresponding partition of the dual canonical basis, which we view as a replacement of the Peter-Weyl decomposition (available in characteristic $0$) over a general base scheme.\par 
For any subset $P\subset X^+$, we define ${}_\A\bfO(P)$ to be the $\A$-submodule of ${}_\A\bfO$ spanned by elements in the set $\bigsqcup\limits_{\la\in P}\dot{\bfB}[\la]^*$ and for any $\A$-algebra $R$ define ${}_R\bfO(P)\coloneqq{}_\A\bfO(P)\otimes_\A R$. 
We say that the set $P$ is downward closed if for any $\la\in P$ and $\mu\in X^+$ with $\mu\le\la$ we have $\mu\in P$. 
\begin{theo}[Corollary \ref{cor:grO-finite}]
    Suppose $P\subset X^+$ is a finitely generated submonoid that is downward closed. Then ${}_R\bfO(P)$ is a finitely generated $R$-subalgebra of ${}_R\bfO$. 
\end{theo}
The basic strategy for the proof is as follows. First we filter ${}_R\bfO(P)$ using the dominance order and reduce finite generation to the corresponding statement for its associated graded algebra. Then we embed the associated graded algebra into the tensor product of two sub-algebras that are described explicitly by the dual canonical basis, for which the finite-generation properties will be more straightforward. It turns out that these subalgebras are precisely the invariants under (the quantization of) the groups $U$ and $U^-$, where $U$ (resp. $U^-$) is the unipotent radical of $B$ (resp. the opposite Borel $B^-$). Therefore we also get an explicit description of these invariant subalgebras in terms of the dual canonical basis, see Proposition \ref{prop:U-invariant} and Corollary \ref{cor:U-inv-O(P)}.\par
In the companion paper \cite{CJ2} we show that when $P$ satisfies some further properties (so that it is a \emph{weight monoid} in the sense of \emph{loc. cit.}), then $\Spec({}_\ZZ\bfO(P))$ is a reductive monoid scheme over $\ZZ$ with unit group $G$. This will prove the existence part of the classification theorem. The uniqueness part says that on a general connected base scheme, any reductive monoid with unit group $G$ is isomorphic to the base change of one of the integral models above. This uniqueness statement will rely on Proposition \ref{prop:U-invariant} and also a technical result that we prove in this paper using the theory of canonical bases, see Proposition \ref{prop:nilpotent-deformation}.\par  
Besides, let us mention some other results that will be used in \cite{CJ2} and may also be of independent interest: Lemma \ref{lem:B-dot-X-ab} proves a simple formula for the product of certain dual canonical basis elements and will be applies when studying very flact reductive monoids; Proposition \ref{prop:iota} and Lemma \ref{lem:B-dot-functorial} establish some basic functorial properties of the canonical bases for certain morphisms of root data, and will be used when studying the decomposition of reductive monoids into groups and monoids with zero; Theorem \ref{theo:O(L)-normal} proves normality of certain classical specialization of the algebras ${}_R\bfO(P)$ and will be used when establishing fiberwise normality of reductive monoids. 

\subsection{Related work} 
In \cite{Ganev}, Ganev constructs the quantized coordinate ring of the Vinberg monoid over $\mathbb{C}$ as the Rees algebra associated to certain canonical filtrations on the quantized coordinate rings of semisimple groups. In \cite{Zhu2020}, generalizing the approach in \cite{XiaoZhu2019} over a field, Zhu defines the Vinberg monoid over $\ZZ$, also using the Rees algebra construction. See also the work of van den Hove \cite{Hove} and Nath \cite{Nath} which are closely related to Zhu's construction. In particular, the canonical bases are not used in these works. On the other hand, our construction of the integral models is closer to the work of Bao and Song~\cite{BaoSongAffineEmbeddings}, in which integral models for affine embeddings of symmetric spaces are constructed using the theory of canonical bases for quantum symmetric pairs developed by Bao and Wang~\cite{Bao-Wang}. Compared to previous works, the new contributions of the current paper are the finite-generation results for the integral models and the explicit description of the quantized integral unipotent invariant subalgebras; these will be crucial ingredients in our classification theorem for split reductive monoids in \cite{CJ2}.

\subsection{Organization of the article}
In \S\ref{sec:quantum-canonical-bases} we fix notation on root data and reductive groups, and review basic facts on Lusztig's modified quantum groups and canonical bases. In \S\ref{sec:Lusztig-group-scheme} we first review Lusztig's construction of Chevalley groups over $\ZZ$ using the dual canonical bases, and then establish some basic functorial properties of the (dual) canonical bases. Then \S\ref{sec:finite-coordinate-ring} is devoted to the proof of finite generation of the quantized integral coordinate rings of reductive monoids. In \S\ref{sec:reps-and-invariants} we use the dual canonical bases to study the (quantization of the) subalgebra of unipotent invariants in the coordinate algebra.

\subsection{Notations and conventions}
The symbol $\NN\coloneqq\ZZ_{\ge0}$ denotes the set of \emph{non-negative} integers. We will often use the Kronecker delta symbol: $\delta_{x,y}=1$ if $x=y$ and $\delta_{x,y}=0$ if $x\ne y$.\par 
Following Lusztig \cite{Lusztig-IntroQuantumGroup}, we let $v$ be an indeterminate and let $\mathcal{A}\coloneqq\ZZ[v, v^{-1}]$ be the Laurent polynomial ring with coefficients in $\ZZ$. Its fraction field is $\QQ(v)$. We will usually use the letter $R$ to denote a general commutative $\A$-algebra. On the other hand, we will usually use $\k$ (resp. $k$) to denote a general commutative ring (resp. field) that is equipped with an $\A$-algebra structure via $v\mapsto1$ (so this is supposed to be a ``classical" object, instead of a general ``quantum" object that we denote by $R$). The rings $\ZZ$ and $\QQ$ are always viewed as $\A$-algebras by mapping $v$ to $1$. All commutative rings (or $\A$-algebras) are assumed to be associative and unital unless specified otherwise. More precisely, the coefficient rings will be commutative and unital, while algebras over them may be non-commutative or non-unital when specified.
\subsection*{Acknowledgment}
This paper is an update of the first part of the first version of the preprint \cite{CJ2}, which was done without using AI tools. During the preparation of the current version, ChatGPT 6 Astra was used for proofreading and text improvements. This research was funded by the Hua Loo-Keng Center for Mathematical Sciences, AMSS, CAS. The research of J.C. was also supported by National Key R\&D Program of China No.2023YFA1009701, National Natural Science Foundation of China (Grant No.12288201, No.12231001), CAS Project for Young Scientists in Basic Research (Grant No.YSBR-033).

\section{Quantized enveloping algebras and canonical bases}\label{sec:quantum-canonical-bases}
\subsection{Review of basic definitions}
\subsubsection{Based root data}\label{subsubsec:based-root-data}
We first recall the standard notion of root data. For our purposes, we will only consider finite type root data.
\begin{defi}\label{def:based-root-data}
   A \emph{based root datum} is a quintuple $\Psi=(\Delta, Y, X,\{\alpha_i,i\in\Delta\}, \{\alpha_i^\vee,i \in\Delta\})$ consisting of:
\begin{itemize}
    \item a finite set $\Delta$;
    \item two finitely generated free abelian groups $X, Y$ and a perfect pairing $\langle \cdot, \cdot \rangle\colon Y \times X \rightarrow \mathbb{Z}$;
    \item an embedding $\Delta\into X$ ($i\mapsto\alpha_i$) and an embedding $\Delta\into Y$ ($i \mapsto \alpha_i^\vee$);
\end{itemize} 
such that the matrix $A=(a_{ij})_{i,j\in\Delta}$, with entries $a_{ij}\coloneqq\langle\alpha_i^\vee,\alpha_j\rangle$ for all $i,j\in\Delta$, is a Cartan matrix.\footnote{We require $A$ to be a Cartan matrix (instead of only a generalized Cartan matrix) since our root data are of finite type. Recall that a Cartan matrix admits a positive diagonal symmetrization that is positive definite. So the images of the embeddings $\Delta\into X$ and $\Delta\into Y$ are linearly independent. Thus a based root datum is both ``$X$-regular and $Y$-regular" in the sense of \cite[\S2.2.2]{Lusztig-IntroQuantumGroup}.}
\end{defi}
For a based root datum $\Psi$ as above, the abelian group $X$ (resp. $Y$) is called the \emph{weight lattice} (resp. \emph{coweight lattice}). The \emph{root lattice}, denoted $\ZZ^\Delta$, is the free abelian group with basis $\Delta$, viewed as subgroup of $X$. The \emph{positive root monoid} is the submonoid $X_{\pos}=\NN^\Delta\subset X$ generated by the simple roots $\Delta$. The set of \emph{dominant weights} of $\Psi$ is defined to be 
\[X^+\coloneqq \{ \mu \in X \mid \langle \alpha_i^\vee, \mu \rangle \geq 0, \forall i \in\Delta\}.\]
It is a submonoid of the weight lattice $X$. Let $X^+_\pos$ be the submonoid of $X$ generated by the dominant weight monoid $X^+$ and the positive root monoid $X_\pos=\NN^\Delta$.\par
For \(\lambda, \lambda'\) in \(X\) we write \(\lambda \geq \lambda'\) or \(\lambda' \leq \lambda\) if \(\lambda - \lambda' \in \sum_{i\in\Delta} \NN\a_i\). The Weyl group $W=W(\Psi)$ is defined to be the finite subgroup of \(\operatorname{Aut}(Y)\) generated by the involutions \(s_i\colon y \mapsto y - \langle y, \a_i \rangle \a_i^\vee, \ (i \in\Delta)\), or equivalently the finite subgroup of \(\operatorname{Aut}(X)\) generated by the involutions \(s_i\colon x \mapsto x - \langle \a_i^\vee, x \rangle \a_i, \ (i \in\Delta)\). Let \(l\colon W \to \NN\) be the standard length function on \(W\) with respect to \(\{s_i \mid i \in\Delta\}\) and let \(w_0 \in W\) be the unique element such that \(l(w_0)\) is maximal.

\begin{defi}\label{def:bilinear-form}
    Let $\{d_i,i\in\Delta\}$ be relatively prime positive integers such that the matrix $(d_ia_{ij})_{i,j\in\Delta}$ is symmetric. There is a unique symmetric bilinear form on $\ZZ^\Delta$, denoted by $\nu,\nu'\mapsto\nu\cdot\nu'$, such that $i\cdot i=2d_i$ and $i\cdot j=d_ia_{ij}$ for all\footnote{Here, in accordance with the conventions in \cite{Lusztig-IntroQuantumGroup}, for simplicity we write $i\cdot j$ instead of $\alpha_i\cdot\alpha_j$} $i,j\in\Delta$. 
    For any $\nu=\sum_{i\in\Delta}\nu_i i\in\ZZ^\Delta$ and any $\la\in X$, we set:
    \[\nu\circ\la\coloneqq\sum_{i\in\Delta}\nu_i\langle\alpha_i^\vee,\la\rangle d_i=\sum_{i\in\Delta}\nu_i\langle\alpha_i^\vee,\la\rangle (i\cdot i/2),\quad\quad\mathrm{tr}(\nu)\coloneqq\sum_{i\in\Delta}\nu_i\in\ZZ.\]
    See \cite[22.1.3]{Lusztig-IntroQuantumGroup}.
\end{defi}

\subsubsection{Notations on reductive groups}\label{sec:notation-reductive-group}
Fix a based root datum 
\[\Psi=(\Delta, Y, X, \{\alpha_i,i\in\Delta\}, \{\alpha_i^\vee,i \in\Delta\})\] 
for the rest of this section. Let us first introduce some notations on reductive groups that will be used throughout the article. We have a split reductive group scheme with pinning $(G,B,T,...)$ over $\Spec(\ZZ)$ with based root datum $\Psi$. Here $T\subset G$ is a split maximal torus and $B\subset G$ is a Borel subgroup containing $T$, all defined over $\Spec(\ZZ)$. By definition we have $X=X^*(T)$, $Y=X_*(T)$ and $W=N_G(T)/T$ is identified with the Weyl group of $T$ in $G$. We will review the construction of $G$ in \S\ref{sec:Lusztig-group-scheme}, following Lusztig's approach. Let $B^-\subset G$ be the Borel subgroup opposite to $B$. Let $U\subset B$ and $U^-\subset B^-$ be the unipotent radicals. For any scheme $S$, we get group schemes $G_S, T_S$ etc. after base change $S$. When $S=\Spec(\k)$ we also write $G_\k, T_\k$ etc. To simplify notations, we will often omit these subscripts when the base scheme is clear from the context. On the other hand, we will sometimes add a subscript $\Psi$ to emphasize the dependence of $G$ on the root data, typically when there are several different root data involved.\par 
We let $G_{\der}$ be the derived group of $G$, $Z_G$ be the center of $G$ and $Z\subset Z_G$ the maximal torus of $Z_G$. Let $T_{\der}\coloneqq T\cap G_{\der}$, $B_{\der}\coloneqq B\cap G_{\der}$ and $Z_{\der}\coloneqq Z\cap G_{\der}$. Then $G_{\der}$ is a split semisimple group scheme over $S$ with maximal torus $T_{\der}$ and Borel subgroup $B_{\der}$. Let $G_{\ab}\coloneqq G/G_{\der}$ be the maximal abelian quotient of $G$. Then $G_{\ab}$ is a split torus over $S$ and we have canonical isomorphisms 
\[G_{\ab}\cong Z/Z_{\der}\cong T/T_{\der}.\] 
Let $G_{\ad}\coloneqq G/Z_G$ be the adjoint group of $G$ and let $T_{\ad}\coloneqq T/Z_G$. Then $X_{\ad}\coloneqq X^*(T_{\ad})$ is identified with the root lattice in $X$ and it contains the positive root monoid: $X_\pos=\NN^\Delta\subset X_\ad$.\par 
Let $X_Z\coloneqq X^*(Z)$, $X_{\ab}\coloneqq X^*(G_{\ab})$ and $X_{\der}\coloneqq X^*(T_{\der})$. Then we have natural embeddings $X_{\ab}\subset X^+\subset X$ and $X_{\ab}\subset X_Z$ that commute with the natural quotient map $X\to X_Z$ (induced by restricting characters along $Z\into T$). From the isomorphism above we get a canonical isomorphism $X_{\der}\cong X/X_{\ab}$. 

\subsubsection{Gaussian binomial coefficients}
We will define various algebras and modules associated to the based root datum $\Psi$ over the Laurent polynomial ring $\A=\ZZ[v,v^{-1}]$ and its fraction field $\QQ(v)$. First we review some standard notations.\par
For $i\in\Delta$ we set $v_i = v^{d_i}$ and we use the following notations
\[\begin{bmatrix}a\\ t\end{bmatrix}_i=\frac{\prod_{s=0}^{t-1}(v_i^{a-s}-v_i^{-a+s})}{\prod_{s=1}^t(v_i^s-v_i^{-s})},\quad a\in\ZZ,t\in\NN,\]
\[[n]_i\coloneqq\begin{bmatrix}n\\1\end{bmatrix}_i=\frac{v_i^n-v_i^{-n}}{v_i-v_i^{-1}},\quad n\in\ZZ,\]
\[[n]_i^!\coloneqq\prod_{s=1}^n[s]_i,\quad n\in\NN.\]
By convention we have
\[[0]_i=0,\quad[0]_i^!=1,\quad\begin{bmatrix}a\\ 0\end{bmatrix}_i=1,\quad\forall a\in\ZZ.\]
These elements, a priori in $\QQ(v)$, actually lie in $\A=\ZZ[v,v^{-1}]$ (see \cite[\S1.3]{Lusztig-IntroQuantumGroup}). Also, the subscript ``$i$" on the left-hand sides will be dropped in case $v_i$ is replaced by $v$ on the right-hand sides.

\subsection{The positive/negative half of quantized enveloping algebra}
Let \(\mathbf{f}\) be the associative \(\QQ(v)\)-algebra with \(1\) defined as in \cite[\S1.2.5]{Lusztig-IntroQuantumGroup}, or equivalently by the generators $\{\theta_i,i\in\Delta\}$ and the quantum Serre relations
\begin{equation}\label{eq:quantum Serre relations}
    \sum_{\substack{p, p' \in \NN \\ p + p' = 1 - \langle\alpha_i^\vee, \alpha_j\rangle}} (-1)^{p'} \left( \frac{\theta_i^p}{[p]_i^!} \right) \theta_j \left( \frac{\theta_i^{p'}}{[p']_i^!} \right)=0,\quad\forall i\ne j\text{ in }\Delta.
\end{equation}
We have a direct sum decomposition \(\mathbf{f} = \bigoplus_{\nu \in \NN^\Delta} \mathbf{f}_{\nu}\) as $\QQ(v)$-vector space such that:
\begin{itemize}
    \item for each $\nu=\sum_{i\in\Delta}\nu_ii\in\NN^\Delta$, the subspace $\bff_\nu$ is finite dimensional and its elements can be represented by linear combinations of words in $\{\theta_i,i\in\Delta\}$ in which \(\theta_i\) occurs $\nu_i$ times for all \(i \in\Delta\).
    \item we have $1\in\bff_0$, $\theta_i\in\bff_i,\forall i\in\Delta$ and for each $\nu,\nu'\in\NN^\Delta$, $\bff_\nu\cdot\bff_{\nu'}\subset\bff_{\nu+\nu'}$.  
\end{itemize}
If $x\in\bff_\nu$, we say that $x$ is homogeneous and write $|x|=\nu$.\par 
For \(i\in\Delta\), \(n \in \ZZ\) we set 
\[\theta_i^{(n)} = \begin{cases}
    \frac{\theta_i^n}{[n]_i^!} \in \mathbf{f} &\text{ if } n \geq 0,\\
    0 &\text{ if }n < 0.
\end{cases}\] 
Let \({}_{\mathcal{A}}\mathbf{f}\) be the unital \(\mathcal{A}\)-subalgebra of \(\mathbf{f}\) generated by the elements \(\theta_i^{(n)}\) with \(i \in\Delta,\ n \in \NN\) (cf. \cite[\S1.4.7]{Lusztig-IntroQuantumGroup}). We have \({}_{\mathcal{A}}\mathbf{f} = \bigoplus_{\nu \in \NN^\Delta} {}_{\A}\mathbf{f}_{\nu}\) where ${}_{\A}\mathbf{f}_{\nu}\coloneqq\mathbf{f}_{\nu}\cap {}_{\mathcal{A}}\mathbf{f}$. 

Let $\sigma\colon\bff\to\bff^{\mathrm{op}}$ be the unique $\QQ(v)$-algebra isomorphism such that $\sigma(\theta_i)=\theta_i$ for all $i\in\Delta$. Clearly $\sigma$ restricts to an $\A$-algebra isomorphism $\sigma\colon{}_\A\mathbf{f}\to{}_\A\mathbf{f}^{\mathrm{op}}$. The \emph{bar involution} is the unique $\QQ$-algebra involution ${\bar{}}\colon\bff\to\bff$ that sends $v$ to $v^{-1}$ and fixes $\theta_i$ for all $i\in\Delta$.
\begin{defi}
    Let $r\colon\bff\to\bff\otimes\bff$ be the unique $\QQ(v)$-linear map such that 
    \[r(\theta_i)=1\otimes\theta_i+\theta_i\otimes1,\quad\forall i\in\Delta\]
    and more generally for any sequence $i_1,\dotsc,i_n$ in $\Delta$ defining a monomial $x=\theta_{i_1}\theta_{i_2}\dotsm\theta_{i_n}\in\bff$,
    \[r(x)=1\otimes x+x\otimes 1+\sum (\prod_{\substack{a\in[1,s],c\in[1,t]\\ j_a>k_c}} v^{i_{j_a}\cdot i_{k_c}})\theta_{i_{j_1}}\dotsm\theta_{i_{j_s}}\otimes\theta_{i_{k_1}}\dotsm\theta_{i_{k_t}}\]
    where the sum runs over all nonempty subsets $j_1<\dotsm<j_s$ of $[1,n]$ with nonempty complement $k_1<\dotsm<k_t$ ($s+t=n$, $s,t\ge1$).\par 
    Let $\bar{r}\colon\bff\to\bff\otimes\bff$ be the unique $\QQ(v)$-linear map given by the same formula as $r$ except $v$ is replaced by $v^{-1}$. For details, see \cite[1.2.6, 1.2.12]{Lusztig-IntroQuantumGroup} and \cite[3.8]{Lusztig-Infinity}.
\end{defi}

Let \(\mathbf{B}\) be the canonical $\QQ(v)$-basis of \(\mathbf{f}\) defined in \cite[§14.4]{Lusztig-IntroQuantumGroup}. It is also an \(\mathcal{A}\)-basis of \({}_{\mathcal{A}}\mathbf{f}\) (by \cite[14.2.3(a)]{Lusztig-IntroQuantumGroup}). Let $\bfB_\nu\coloneqq\mathbf{f}_\nu\cap\bfB$ for any $\nu\in\NN^\Delta$. Then $\bfB=\bigsqcup_{\nu\in\NN^\Delta}\bfB_\nu$ and by \cite[14.4.3(c)]{Lusztig-IntroQuantumGroup} we have $\sigma(\bfB_\nu)=\bfB_\nu$.\par 
Given $i\in\Delta$ and $n\ge0$, we let $\bfB_{i;\ge n}\coloneqq\bfB\cap\theta_i^n\bff$ and ${}^\sigma\bfB_{i;\ge n}\coloneqq\bfB\cap\bff\theta_i^n$. Let $\bfB_{i;n}\coloneqq\bfB_{i;\ge n}\setminus\bfB_{i;\ge n+1}$ and ${}^\sigma\bfB_{i;n}\coloneqq{}^\sigma\bfB_{i;\ge n}\setminus{}^\sigma\bfB_{i;\ge n+1}$. \par 
For any $\la\in X^+$, we define
\[\bfB(\la)\coloneqq\bigcap_{i\in\Delta}\left(\bigcup_{0\le n\le\langle\alpha_i^\vee,\la\rangle}{}^\sigma\bfB_{i;n}\right).\]
\begin{lem}\label{lem:B-union}
    Let $\L\subset X^+$ be a subset with the following property: for any $N\ge0$ there exists $\la\in\L$ such that $\langle\alpha_i^\vee,\la\rangle\ge N$ for all $i\in\Delta$. Then we have $\bfB=\bigcup_{\la\in\L}\bfB(\la)$. 
\end{lem}
\begin{proof}
    For any $i\in\Delta$ we have a partition $\bfB=\bigsqcup_{n\ge0}{}^\sigma\bfB_{i;n}$. Thus for any $b\in\bfB$ and $i\in\Delta$ there is an integer $m(b,i)\ge0$ such that $b\in{}^\sigma\bfB_{i;m(b,i)}$. By assumption there exists $\la\in\L$ such that $\langle\alpha_i^\vee,\la\rangle\ge m(b,i)$ for all $i\in\Delta$ and then we have $b\in\bfB(\la)$.
\end{proof}
By \cite[14.2.3]{Lusztig-IntroQuantumGroup}, the canonical basis $\bfB$ is characterized up to sign by the bar involution and a non-degenerate bilinear form $(\cdot,\cdot)\colon\bff\times\bff\to\QQ(v)$ (defined in \cite[1.2.3, 1.2.5]{Lusztig-IntroQuantumGroup}) as follows:
\begin{equation}\label{eq:pairing characterizing B}
    \pm\bfB=\{x\in{}_\A\bff\mid \bar{x}=x, (x,x)\in1+v^{-1}\ZZ[[v^{-1}]]\cap\QQ(v)\}.
\end{equation}

\subsection{The modified quantized enveloping algebra}
The Drinfeld-Jimbo quantized enveloping algebra $\bfU$ (associated to the based root datum $\Psi$) is a $\QQ(v)$-algebra generated by two copies of $\bff$ (denoted by $\bfU^+$ and $\bfU^-$), together with the group algebra of the \emph{coweight} lattice $Y$. See \cite[\S3]{Lusztig-IntroQuantumGroup} for the precise definition, which we do not reproduce here since we will not use it. Instead we will need Lusztig's modified form of $\bfU$, denoted $\dot{\bfU}$, which (after suitable completion) also contains the two copies ($\bfU^{+},\bfU^-$) of $\bff$, but with the group algebra of $Y$ in $\bfU$ replaced by the algebra of functions with finite support on the \emph{weight} lattice $X$ (i.e. the $\QQ(v)$-algebra $\bigoplus_{\zeta\in X}\QQ(v)1_\zeta$ with basis $\{1_\zeta,\zeta\in X\}$). The algebra $\dot{\bfU}$ is simpler than $\bfU$ in several aspects: it captures only the most essential representations of $\bfU$ (the so-called ``type $1$" weight modules); also one can construct its integral model (an $\A$-algebra ${}_{\A}\dot{\mathbf{U}}$ such that ${}_{\A}\dot{\mathbf{U}}\otimes_\A\QQ(v)=\dot{\bfU}$) more easily.\par 

Following \cite[\S31.1.3]{Lusztig-IntroQuantumGroup}, we define the \(\A\)-algebra  \({}_{\A}\dot{\mathbf{U}}\) (without using $\bfU$) as follows. It is generated by the symbols \( x^{+}1_{\zeta}x^{\prime -} \), \( x^{-}1_{\zeta}x^{\prime +} \) with \( x \in {}_{\A}\mathbf{f}_{\nu} \), \( x^{\prime} \in {}_{\A}\mathbf{f}_{\nu^{\prime}} \) for various \(\nu, \nu^{\prime}\) and \(\zeta \in X\); these symbols are subject to the following relations (if \( x \) or \( x' \) in \( x^+ 1_\zeta x^{\prime -} \) or \( x^- 1_\zeta x^{\prime +} \) is \( 1 \), we omit writing it):
\begin{itemize}
    \item The maps \(_{\A}\mathbf{f} \rightarrow {}_{\A}\dot{\mathbf{U}}\), \( x \mapsto x^{\pm}1_{\zeta} \) are \(\A\)-linear for any \(\zeta \in X\);
    \item For any $m, n\in\NN$, $\zeta \in X$ and $i \neq j$ in $\Delta$, 
    \[\theta^{(n)+}_{i}1_{\zeta}\theta^{(m)-}_j = \theta^{(m)-}_j1_{\zeta+n\alpha_i+m\alpha_j}\theta^{(n)+}_i;\]
    \item  For any \( m, n \in \NN \), \(\zeta \in X\) and \( i \in \Delta\),
    \[\theta^{(n)\pm}_{i}1_{\mp\zeta}\theta^{(m)\mp}_{i} = \sum_{t\in \NN}\begin{bmatrix} m + n - \langle \a_i^\vee, \zeta \rangle \\ t \end{bmatrix}_{i} \theta^{(m-t)\mp}_{i}1_{\mp\zeta \pm (n+m-t)\a_i} \theta^{(n-t)\pm}_{i};\]
    \item For any \( x \in {}_{\A}\mathbf{f}_{\nu} \) and \( \zeta \in X \),
    \[ x^\pm 1_\zeta = 1_{\zeta \pm \nu} x^\pm;\]
    \item  For any \( x, x' \in {}_{\A}\mathbf{f} \) and \( \zeta \in X \),
    \[(x^\pm 1_\zeta)(1_{\zeta'} x^{\prime \mp}) = \delta_{\zeta, \zeta'} x^\pm 1_\zeta x^{\prime \mp};\]
    \item  For any \( x \in {}_{\A}\mathbf{f}_{\nu} \), \( x' \in {}_{\A}\mathbf{f} \) and \( \zeta \in X \),
    \[(x^\pm 1_\zeta)(1_{\zeta'} x^{\prime \pm}) = \delta_{\zeta, \zeta'} 1_{\zeta \pm \nu} (x x')^\pm.\]
\end{itemize}
 The algebra \({}_{\A}\dot{\mathbf{U}}\) does not have a unit element in general; instead, it has a family of idempotents \( 1_\lambda \) (\( \lambda \in X \)) such that $1_\lambda 1_{\lambda'} = \delta_{\lambda, \lambda'} 1_\lambda$ for any $\lambda, \lambda' \in X$. They induce a decomposition
\begin{equation}\label{eq:U-dot-decomposition}
     {}_{\A}\dot{\mathbf{U}} = \bigoplus_{\lambda, \lambda' \in X} 1_\lambda ({}_{\A}\dot{\mathbf{U}}) 1_{\lambda'}.
\end{equation}
From the defining relations above we see that $1_\lambda ({}_{\A}\dot{\mathbf{U}}) 1_{\lambda'}\ne0$ only if $\la-\la'\in\ZZ^\Delta$. \par
There are two natural ${}_{\A}\mathbf{f}\otimes{}_{\A}\mathbf{f}^{\mathrm{op}}$-module structures on ${}_{\A}\dot{\mathbf{U}}$ giving rise to two isomorphisms of \( \A \)-modules: 
\[{}_{\A}\mathbf{f} \otimes  \A[X] \otimes{}_{\A}\mathbf{f} \to {}_{\A}\dot{\mathbf{U}}, \quad x \otimes  \lambda \otimes  x' \mapsto x^{\pm} 1_\lambda x^{\prime \mp}.\]
Moreover, the $\A$-algebra ${}_{\A}\dot{\mathbf{U}}$ has extra structures that resemble (but not literally the same as) a Hopf algebra. The \emph{coproduct} on $ {}_{\A}\dot{\mathbf{U}}$ is given, under the decomposition \eqref{eq:U-dot-coproduct}, by a family of $\A$-linear maps
\begin{equation}\label{eq:U-dot-coproduct}
    \Delta\colon {}_{\A}\dot{\mathbf{U}}\to\prod_{\la_1,\la_1',\la_2,\la_2'\in X}1_{\la_1} {}_{\A}\dot{\mathbf{U}}1_{\la_1'}\otimes_\A 1_{\la_2} {}_{\A}\dot{\mathbf{U}}1_{\la_2'}
\end{equation}
defined as in \cite[\S23.1.5, \S23.2.3]{Lusztig-IntroQuantumGroup}. Here the right hand side contains the tensor product ${}_{\A}\dot{\mathbf{U}}\otimes_\A {}_{\A}\dot{\mathbf{U}}$ as a subalgebra. We recall that in \emph{loc. cit.} Lusztig uses the coproduct on the (usual) quantum enveloping algebra $\bfU$ to define maps
\[\Delta_{\la_1,\la_1';\la_2,\la_2'}\colon\;\;{}_{\la_1+\la_2}\dot{\bfU}_{\la_1'+\la_2'}\to{}_{\la_1}\dot{\bfU}_{\la_1'}\otimes{}_{\la_2}\dot{\bfU}_{\la_2'},\quad\forall\la_1,\la_2,\la_1',\la_2'\in X\]
which preserve the $\A$-structures and combine into the homomorphism $\Delta$ above. As explained in \cite[\S3.9]{Lusztig-Infinity}, the coproduct $\Delta$ can also be defined directly (without recourse to $\bfU$): it is uniquely characterized by the following requirements:
\begin{enumerate}
    \item For any homogeneous elements $x\in\bff$ with $r(x)=\sum_j(x_j\otimes x_j')$, $\bar{r}(x)=\sum_k({}_kx\otimes{}_kx')$ where $x_j,x_j',{}_kx,{}_kx'\in\bff$ are homogeneous, we have (see Definition \ref{def:bilinear-form} for the notations)
    \begin{align}\label{eq:coproduct-formula}
        \Delta(x^+)&=\sum_{j;\la\in X}v^{|x_j'|\circ\la}x_j^+1_\la\otimes (x_j')^+, \\
        \Delta(x^-)&=\sum_{k;\la\in X}v^{-|{}_kx|\circ\la}({}_kx)^-\otimes1_\la({}_kx')^-.
    \end{align}
    \item For any $\la\in X$ we have
    \[\Delta(1_\la)=\sum_{\la'\in X}1_{\la'}\otimes 1_{\la-\la'}.\]
\end{enumerate}

The \emph{counit} of ${}_{\A}\dot{\mathbf{U}}$ is the unique $\A$-algebra homomorphism such that
\[\epsilon(1_\zeta)=\delta_{\zeta,0},\quad\forall \zeta\in X;\quad\epsilon(\theta_i)=0,\quad\forall i\in\Delta.\]
The \emph{antipode} of ${}_\A\dot{\bfU}$ is an $\A$-algebra isomorphism $S:{}_\A\dot{\bfU}\to{}_A\dot{\bfU}^{\mathrm{op}}$ given explicitly in \cite[23.1.7]{Lusztig-IntroQuantumGroup}.

We recall the following isomorphisms (see \cite[\S23.1.6, §31.1.4]{Lusztig-IntroQuantumGroup}):
\begin{itemize}
    \item The $\A$-algebra isomorphism $\sigma\colon{}_\A\dot{\bfU}\to{}_A\dot{\bfU}^{\mathrm{op}}$ defined by
\[\sigma(\theta_i^+)=\theta_i^+,\quad \sigma(\theta_i^-)=\theta_i^-,\quad \sigma(1_\lambda)=1_{-\lambda}.\]
    \item The $\A$-algebra automorphism $\omega\colon {}_{\A}\dot{\mathbf{U}} \to {}_{\A}\dot{\mathbf{U}}$ defined by
\[\omega(\theta_i^+)=\theta_i^-,\quad\omega(\theta_i^-)=\theta_i^+,\quad\omega(1_\lambda)=1_{-\lambda}.\]
    \item The semi-linear $\A$-algebra automorphism $\bar{}:{}_\A\dot{\bfU}\to{}_\A\dot{\bfU}$, called the \emph{bar involution}, defined by requiring that the elements $\theta_i^{\pm}1_\zeta$ for $i\in\Delta,\zeta\in X$ are all fixed and
    \[\overline{a\cdot u}=\bar{a}\cdot\bar{u},\quad\forall a\in\A, u\in{}_\A\dot{\bfU}\]
    where $\bar{a}(v)\coloneqq a(v^{-1})$ for any $a(v)\in\A=\ZZ[v,v^{-1}]$. In particular, all the Gaussian binomial coefficients $[n]_i^!$ and all the elements $\theta_i^{(n)^{\pm}}1_\zeta$ are fixed under the bar involution.
\end{itemize}
We remark that the antipode $S$ differs from $\sigma$ by a factor which is the product of a sign and a power of $v$. Also, we have $\sigma^2=1$ and $\omega^2=1$ but in general $S^{-1}$ is not equal to $S$. 

\begin{defi}
    For any $\A$-algebra $R$, we denote ${}_{R}\dot{\mathbf{U}}\coloneqq{}_{\A}\dot{\mathbf{U}}\otimes_{\A} R $. When $R=\QQ(v)$ we simply write $\dot{\bfU}={}_{\QQ(v)}\dot{\bfU}$.
\end{defi}

Let \(\dot{\mathbf{B}}\) be the canonical basis of \({}_{\A}\dot{\mathbf{U}}\) defined in \cite[25.2]{Lusztig-IntroQuantumGroup}. Let us review its construction and basic properties. 
\begin{defi}\label{def:B-dot}
    For any $\eta,\zeta\in X$, denote ${}_\eta\dot{\bfB}_\zeta\coloneqq\dot{\bfB}\cap 1_\eta\dot{\bfU}1_\zeta$. By construction $\dot{\bfB}$ admits a partition:
    \[\dot{\bfB}=\bigsqcup_{\eta,\zeta\in X}{}_\eta\dot{\bfB}_\zeta.\]
    The canonical basis $\dot{\bfB}$ is in bijection with $\bfB\times\bfB\times X$: an element of $\dot{\bfB}$ corresponding to $(b_1,b_2,\zeta)$ is denoted by $b_1\diamondsuit_\zeta b_2$, and it equals $b_1^+b_2^-1_\zeta$ plus an $\A$-linear combination of elements of the form $c_1^+c_2^-1_\zeta$ with $c_1,c_2\in\bfB$ satisfying the following conditions:
    \begin{itemize}
        \item $\mathrm{tr}|c_1|<\mathrm{tr}|b_1|$ and $\mathrm{tr}|c_2|<\mathrm{tr}|b_2|$,
        \item $|c_1|-|c_2|=|b_1|-|b_2|$.
    \end{itemize}
    See \cite[Proof of Theorem 25.2.1]{Lusztig-IntroQuantumGroup}. In particular, we have $b_1\diamondsuit_\zeta b_2\in{}_\eta\dot{\bfB}_\zeta$ where $\eta=\zeta-|b_2|+|b_1|$.
\end{defi}
For any $\A$-algebra $R$ and any \(b \in \dot{\mathbf{B}}\), we shall denote the element \(b \otimes_{\mathcal{A}}1\in {}_{\mathcal{A}}\dot{\mathbf{U}}\otimes_{\A}R={}_R\dot{\mathbf{U}}\) also by \(b\). Then \(\dot{\bfB}\) can also be viewed as an $R$-basis of \({}_R\dot{\mathbf{U}}\). \par

The structure constants of ${}_{\A}\dot{\mathbf{U}}$ with respect to the basis $\dot{\bfB}$ are the family of elements $m_{ab}^c,\hat{m}^{ab}_c\in\A$ for any $a,b,c\in\dot{\bfB}$ defined by:
\[ab=\sum_{c\in\dot{\bfB}}m_{ab}^c c,\quad\forall a,b\in\dot{\bfB},\]
\[\Delta (c)=\sum_{a,b\in\dot{\bfB}}\hat{m}_c^{ab}a\otimes b,\forall c\in\dot{\bfB}.\]
We refer to \cite[\S25.4]{Lusztig-IntroQuantumGroup} for a list of identities satisfied by the structure constants. We record some finiteness properties.
\begin{lem}\label{lem:structure-constant-finite}
\begin{enumerate}
    \item[(i)] For any $c\in\dot{\bfB}$ the set $\{(a,b)\in\dot{\bfB}\times\dot{\bfB}\mid m_{ab}^c\ne0\}$ is finite.
    \item[(ii)] For any $a,b\in\dot{\bfB}$, the sets $\{c\in\dot{\bfB}\mid\hat{m}_c^{ab}\ne0\}$ and $\{c\in\dot{\bfB}\mid m^c_{ab}\ne0\}$ are both finite.
\end{enumerate}
\end{lem}
\begin{proof}
    (i) This is \cite[Lemma 1.8]{LusztigGroupScheme}.\par 
    (ii) The statement for $\hat{m}^{ab}_c$ is \cite[Lemma 1.16]{LusztigGroupScheme}. The statement for $m_{ab}^c$ is clear since any element of ${}_\A\dot{\bfU}$ (in particular $ab$) can be written as a finite $\A$-linear combination of elements in $\dot{\bfB}$. 
\end{proof}

The canonical basis $\dot{\mathbf{B}}$ admits a partition indexed by the dominant weights \cite[§29.1.1]{Lusztig-IntroQuantumGroup}
$$\dot{\mathbf{B}}=\bigsqcup_{\lambda \in X^+} \dot{\mathbf{B}}[\lambda]$$ such that each component is a finite set (by \cite[\S29.1.6]{Lusztig-IntroQuantumGroup}), called the \emph{two-sided cells} of $\dot{\bfB}$. For any subset $P\subset X^+$, we denote 
$$\dot{\mathbf{B}}[P]\coloneqq\bigsqcup_{\lambda\in P} \dot{\mathbf{B}}[\lambda]$$
and we let $\dot{\bfU}[P]$ be the subspace of $\dot{\bfU}$ spanned by elements in $\dot{\bfB}[P]$. In particular, for any $\la\in X^+$, we have subsets
\[X^+_{\ge\la}\coloneqq\{\mu\in X^+\mid \mu\ge\la\},\quad X^+_{>\la}\coloneqq\{\mu\in X^+\mid \mu>\la\}\]
and we denote
\[\dot{\bfU}[\ge\la]\coloneqq\dot{\bfU}[X^+_{\ge\la}],\quad\dot{\bfU}[>\la]\coloneqq\dot{\bfU}[X^+_{>\la}].\]
\begin{prop}\label{prop:B-dot-involution}
    For any $\lambda\in X^+$, we have $\sigma(\dot{\bfB}[\lambda])=\dot{\bfB}[-w_0(\lambda)]$ and $\omega(\dot{\bfB}[\lambda])=\dot{\bfB}[-w_0(\lambda)]$. 
\end{prop}
\begin{proof}
    In \cite{Lusztig-IntroQuantumGroup}, this result is proved up to signs. The more refined statement will follow by combining a result of Kashiwara, as explained in \cite[3.8]{Lusztig-semifield}. More precisely, by \cite[Theorem 4.3.2(ii)]{Kashiwara-modified} we have $\sigma(\dot{\bfB})=\dot{\bfB}$ (note that in \emph{loc.cit.} the involution $\sigma$ is denoted by ``$*$"). Then by \cite[29.3.1(a)]{Lusztig-IntroQuantumGroup} we get that $\sigma(\dot{\bfB}[\lambda])=\dot{\bfB}[-w_0(\lambda)]$. On the other hand, by \cite[Lemma 4.14, Proposition 4.4]{Lusztig-Infinity}, we have $\sigma\omega\dot{\bfB}[\lambda]=\dot{\bfB}[\lambda]$ (in \emph{loc. cit.}, the composition $\sigma\omega$ is denoted by ``$\#$"). Then it follows that $\omega(\dot{\bfB}[\lambda])=\dot{\bfB}[-w_0(\lambda)]$. See also \cite[3.16(a)]{Lusztig-semifield} for the last equality.
\end{proof}

We recall the following parametrization of the two-sided cells $\dot{\bfB}[\la]$.
\begin{prop}[see \cite{Lusztig-Infinity} Proposition 4.4]\label{prop:parametrize-B[la]}
    For any $\la\in X^+$, there is a unique bijection 
    \[\beta_\la\colon\bfB(\la)\times\bfB(\la)\to\dot{\bfB}[\la]\] 
    satisfying
    \[\beta_\la(b_1,b_2)-b_1^-1_\la\sigma(b_2)^+\in\dot{\bfU}[>\la],\quad\forall b_1,b_2\in\bfB(\la).\]
\end{prop}
Similar to $\bfB$, the canonical basis $\dot{\bfB}$ can be characterized (up to sign) by a certain bilinear form on $\dot{\bfU}$. To explain this we introduce more notations. For any $x\in\dot\bfU$, denote $x^\#\coloneqq\sigma\omega(x)=\omega\sigma(x)$ so that $x\mapsto x^\#$ defines an isomorphism of $\A$-algebras ${}_\A\dot{\bfU}\to{}_\A\dot{\bfU}^{\mathrm{op}}$ that exchanges the two weight indices in the decomposition \eqref{eq:U-dot-decomposition}. For any $\la,\la'\in X$ with $\nu\coloneqq\la-\la'\in\ZZ^\Delta$ and any $x\in1_\lambda {}_{\A}\dot{\mathbf{U}} 1_{\lambda'}$, let $\rho(x)=v^{\nu\circ(\la+\la')/2}x^\#$. This defines an $\A$-algebra isomorphism $\rho\colon{}_\A\dot{\bfU}\to{}_\A\dot{\bfU}^{\mathrm{op}}$ such that $\rho^2=1$. See \cite[\S3.7]{Lusztig-Infinity} and compare \cite[19.1.1]{Lusztig-IntroQuantumGroup} (in the latter reference the involution $\rho$ is defined for the usual quantized enveloping algebra $\bfU$). 

\begin{prop}\label{prop:bilinear pairing characterizing dotB}
    There exists a unique symmetric bilinear pairing $(\cdot,\cdot)\colon\dot{\bfU}\times\dot{\bfU}\to\QQ(v)$ satisfying:
    \begin{enumerate}
        \item The decomposition \eqref{eq:U-dot-decomposition} is orthogonal with respect to $(\cdot,\cdot)$;
        \item $(ux,y)=(x,\rho(u)y)$ for all $x,y,u\in\dot{\bfU}$;
        \item For all $x,x'\in\bff$ and $\la\in X$ we have $(x^-1_\la,x'^-1_\la)=(x,x')$.
    \end{enumerate}
    Moreover, we have
    \[\pm\dot{\bfB}=\{\beta\in{}_\A\dot{\bfU}\mid \bar{\beta}=\beta,(\beta,\beta)\in1+v^{-1}\QQ[[v^{-1}]]\cap\QQ(v)\}.\]
\end{prop}
\begin{proof}
    The first statement is \cite[Theorem 26.1.2]{Lusztig-IntroQuantumGroup} and the second statement is \emph{loc. cit. }Theorem 26.3.1. In the notation there, we have $\mathbf{A}=\QQ[v^{-1}]]\cap\QQ(v)$, see \emph{loc. cit.} \S14.2.1.
\end{proof}

\subsection{Highest weight modules}
Recall the following notion from \cite[3.4.1, 31.1.5]{Lusztig-IntroQuantumGroup}:

\begin{defi}\label{defi:unital Udotmodule}
    Let $R$ be a commutative $\A$-algebra. A ${}_{R}\dot{\mathbf{U}}$-module $M$ is said to be \emph{unital} if for any $m\in M$, we have $1_{\lambda}\cdot m=0$ for all but finitely many $\lambda\in X$, and $\sum_{\lambda\in X}1_{\lambda}\cdot m=m$. Let ${}_R\dot{\bfU}\text{-}\mod$ denote the category of \emph{unital} ${}_R\dot{\bfU}$-modules (we will not consider non-unital ${}_R\dot{\bfU}$ modules) and let ${}_R\dot{\bfU}\text{-}\mod^{ft}$ be its full subcategory consisting of modules that are finitely generated as $R$-modules. 
\end{defi}

Let us recall the construction of the basic building blocks of the category ${}_R\dot{\bfU}\text{-}\mod$.\par
Let $\lambda \in X^+$ be a dominant integral weight. Let $\Lambda_\lambda$ be the $\QQ(v)$-vector space $\mathbf{f}/T_\lambda$ where $T_\lambda\coloneqq\sum_{i\in\Delta} \mathbf{f}\theta_i^{\langle\a_i^{\vee},\lambda\rangle+1}$. Let $\eta_\lambda = 1 + T_\lambda\in\La_\la$. The modified form ${}_{\A}\dot{\mathbf{U}}$ acts naturally on $\Lambda_\lambda$ through $\theta_i^+1_{\lambda'}\cdot \eta_\lambda = 0$ for $i\in\Delta, \lambda' \in X$ and $x^{-}1_{\lambda'}\cdot \eta_\lambda = \delta_{\lambda,\lambda'} x + T_\lambda$ for $x \in \mathbf{f}, \lambda' \in X$. Let
$${}_{\A}\Lambda_{\lambda} \coloneqq{}_{\A}\dot{\mathbf{U}}\cdot \eta_\lambda\subset\La_\la$$
be the ${}_{\A}\dot{\mathbf{U}}$-submodule of $\La_\la$ generated by the highest weight vector $\eta_\la$. For any $\A$-algebra $R$, denote
${}_{R}\Lambda_{\lambda}\coloneqq{}_{\A}\Lambda_{\lambda}\otimes_{\A}R$. We recall some basic properties of $\La_\la$. 
\begin{theo}\label{theo:Weyl module}\cite{Lusztig-IntroQuantumGroup}
Let $R$ be a commutative $\A$-algebra.
\begin{enumerate}
    \item The $\QQ(v)$-vector space $T_\lambda$ is a coordinate space for $\mathbf{B}$ and the map $b\mapsto b\eta_\la$ defines a bijection from $\bfB(\la)$ onto a basis of $\La_\la$. 
    \item The module ${}_R\Lambda_{\lambda}$ is free as an $R$-module and belongs to the category ${}_R\dot{\bfU}\text{-}\mod^{ft}$.
    \item The weights of ${}_R\Lambda_{\lambda}$ lie in the set $\{\mu\in X\mid w_0(\lambda)\leq \mu\leq \lambda\}$ and the highest weight space is the rank-one free $R$-module with basis $\eta_\la$.
\end{enumerate}
\end{theo}
\begin{proof}
    The assertion (1) is \cite[Theorem 14.4.11]{Lusztig-IntroQuantumGroup}. Together with \cite[3.4.1, 31.1.5]{Lusztig-IntroQuantumGroup}, it implies (2). The assertion (3) is a consequence of the construction of $\Lambda_{\lambda}$, and is also described in \cite[18.1.1]{Lusztig-IntroQuantumGroup}.
\end{proof}

The following result gives a characterization of the two-sided cells $\dot{\mathbf{B}}[\lambda]$ in terms of the modules $\Lambda_{\lambda}$.
\begin{lem}\label{lem:characterization subset Blambda}
    For any $\lambda\in X^+$, an element $b\in \dot{\mathbf{B}}$ lies in $\dot{\bfB}[\la]$ if and only if the following conditions are satisfied:
    \begin{enumerate}
        \item  $b\cdot\Lambda_{\lambda}\neq 0$, and 
        \item if $b\cdot\Lambda_{\tau}\neq 0$ for some $\tau \in X^+$, then $\tau\geq\lambda$.
    \end{enumerate}
    In particular,  we have 
 \begin{equation}\label{eq:Blambda inside leq lambda}
     \Set{b\in \dot{\mathbf{B}}|b\cdot\Lambda_{\lambda}\neq 0}\subset \dot{\mathbf{B}}[\leq \lambda].
 \end{equation}
\end{lem}
\begin{proof}
    This is a consequence of \cite[Lemma 29.1.3, 29.1.4]{Lusztig-IntroQuantumGroup}. 
\end{proof}
\begin{cor}\label{cor:B[la]-W-la}
    For any $\la\in X^+$, the set $\{\mu\in X|1_{\mu}\in \dot{\mathbf{B}}[\lambda]\}$ coincides with $W\cdot\la$, the Weyl group orbit of $\la$. 
\end{cor}
\begin{proof}
    For any $\mu\in X$, the element $1_{\mu}$ acts on a $\dot{\bfU}$-module as the projection onto the $\mu$-weight space. Thus by Lemma \ref{lem:characterization subset Blambda},  we have $1_{\mu}\in\dot{\mathbf{B}}[\lambda]$ if and only if $\mu$ is a weight for $\Lambda_{\lambda}$ and whenever it is a weight for $\Lambda_{\tau}$ for some $\tau\in X^+$ then $\tau\geq\lambda$.
    Since the Weyl group $W$ permutes the weights of any integrable (in particular, finite dimensional) $\dot{\bfU}$-module (see \cite[Proposition 5.2.7]{Lusztig-IntroQuantumGroup}), the set  $\{\mu\in X|1_{\mu}\in \dot{\mathbf{B}}[\lambda]\}$ is a union of $W$-orbits. According to the description of the weights of $\La_\la$ (see Theorem \ref{theo:Weyl module}), the only dominant weight in this set is $\lambda$. Since any $W$-orbit in $X$ meets $X^+$, we see that the set is precisely the $W$-orbit of $\la$.
\end{proof}

\begin{lem}\label{lem:B(lambda)-characterization}
    Let $\lambda\in X^+$ and $b\in\bf{B}$. The following are equivalent:
    \begin{enumerate}
        \item[(i)] $b\in\bf{B}(\lambda)$;
        \item[(ii)] $b^-1_\lambda\in\dot{\bf{B}}[\lambda]$;
        \item[(iii)] $1_\lambda\sigma(b)^+\in\dot{\bf{B}}[\lambda]$;
        \item[(iv)] $b^+1_{-\lambda}\in\dot{\bf{B}}[-w_0(\lambda)]$.
    \end{enumerate}
\end{lem}
\begin{proof}
The implications (i)$\Rightarrow$(ii) and (i)$\Rightarrow$(iii) follow from \cite[Lemma 4.5]{Lusztig-Infinity}. \par 
(ii)$\Rightarrow$(i): By \cite[29.1.6]{Lusztig-IntroQuantumGroup}, the element $b^-1_\lambda\in\dot{\bf{B}}[\lambda]$ does not act by zero on $\Lambda_\lambda$. This implies that $b\eta_\lambda\ne0$ and hence $b\in\bfB(\lambda)$. \par 
Finally the equivalences (ii)$\Leftrightarrow$(iii)$\Leftrightarrow$(iv) follow from Proposition \ref{prop:B-dot-involution}. 
\end{proof}

Besides the two-sided cells, there is another family of finite subsets of $\dot{\bfB}$ that will be useful. For any $\A$-algebra $R$ and ${}_R\dot{\bfU}$-module $M$, we let ${}^\omega M$ be the ${}_R\dot{\bfU}$-module with the same underlying $R$-module as $M$ and with $u\in{}_R\dot{\bfU}$ acting on ${}^\omega M$ by the operator $\omega(u)$ on $M$. Recall from \cite[\S1.9]{LusztigGroupScheme} that for any $\lambda,\lambda'\in X^+$ there is a unique subset $\dot{\bf{B}}_{\lambda,\lambda'}\subset\dot{\bf{B}}$ mapping bijectively to a basis of ${}^\omega\Lambda_\lambda\otimes \Lambda_{\lambda'}$ under the map $a\mapsto a(\xi_{-\lambda}\otimes \eta_{\lambda'})$, where $\xi_{-\la}\in{}^\omega\La_\la$ is the lowest-weight vector corresponding to $\eta_\la$. We have an explicit description of these sets in the following special cases: 
\[\dot{\bfB}_{\lambda,0}=\{b^+1_{-\lambda};b\in\bfB(\lambda)\},\quad\dot{\bfB}_{0,\lambda}=\{b^-1_{\lambda};b\in\bfB(\lambda)\}=\omega(\dot{\bfB}_{\lambda,0}).\]
(See \cite[\S3.4]{Lusztig-IntroQuantumGroup}. In the notation of \emph{loc.cit.}, $\bfB_\la=\bfB_\la'=\bfB(\la)$.)\par 
 By \cite[Lemma 4.5]{Lusztig-Infinity} we have $\dot{\bfB}_{\lambda,0}\subset\sigma(\dot{\bfB}[\lambda])=\dot{\bfB}[-w_0(\lambda)]$ and $\dot{\bfB}_{0,\lambda}\subset\dot{\bfB}[\lambda]$. In particular for any $\lambda,\lambda'\in X^+$ with $\lambda\ne\lambda'$ we have
\[\dot{\bfB}_{\lambda,0}\cap\dot{\bfB}_{\lambda',0}=\varnothing,\quad\dot{\bfB}_{0,\lambda}\cap\dot{\bfB}_{0,\lambda'}=\varnothing.\]

\section{Lusztig's construction of the Chevalley group schemes}\label{sec:Lusztig-group-scheme}
In this section we first review Lusztig's construction of the quantized coordinate ring of the split reductive group scheme from \S\ref{sec:notation-reductive-group}. Then we study the canonical filtration on the quantized coordinate ring determined by the two-sided cells of the canonical bases. Finally in \S\ref{subsec:functorial-properties} we prove some basic functorial properties of the canonical bases.\par 
We keep the notations from \S\ref{sec:quantum-canonical-bases}. In particular, we fix a based root datum $\Psi$ as in \S\ref{subsubsec:based-root-data}. 

\subsection{Dual canonical bases and quantized coordinate rings}
\begin{defi}
    For any \(a \in \dot{\mathbf{B}}\) we define an $\A$-linear map $a^*\colon{}_{\A}\dot{\mathbf{U}}\to\A$ by $a' \mapsto \delta_{a,a'}$ for all $a' \in \dot{\mathbf{B}}$. In particular,  $1_0^*=\epsilon$ is the counit of ${}_\A\dot{\bfU}$. Then $\dot{\bfB}^*\coloneqq\{a^*, a\in\dot{\bfB}\}$ is an $\A$-linearly independent subset of $\mathrm{Hom_{\A-mod}}({}_\A\dot{\mathbf{U}},\A)$ called the \emph{dual canonical basis}.\par 
    For any $\la\in X^+$, denote $\dot{\bfB}[\la]^*\coloneqq\{b^*, b\in\dot{\bfB}[\la]\}$. Then we get a partition
    \[\dot{\bfB}^*=\bigsqcup_{\la\in X^+}\dot{\bfB}[\la]^*.\]
    More generally, for any subset $P\subset X^+$, we denote 
    \[\dot{\bfB}[P]^*\coloneqq\bigsqcup_{\la\in P}\dot{\bfB}[\la]^*.\]
\end{defi}

Let ${}_\A\bfO$ be the free $\A$-submodule of $\mathrm{Hom_{\A-mod}}({}_\A\dot{\mathbf{U}},\A)$ spanned by $\dot{\bfB}^*$. It inherits a natural Hopf algebra structure from ${}_\A\dot{\bfU}$ with 
\begin{itemize}
    \item unit element $1_0^*$;
    \item multiplication rule (well-defined by Lemma \ref{lem:structure-constant-finite})
    \[a^*b^*=\sum_{c\in\dot{\bfB}}\hat{m}_c^{ab}c^*,\quad\forall a,b\in\dot{\bfB};\]
    \item comultiplication rule (well-defined by Lemma \ref{lem:structure-constant-finite})
    \[c^*\mapsto\sum_{a,b\in\dot{\bfB}}m_{ab}^c a^*\otimes b^*,\quad\forall c\in\dot{\bfB};\]
    \item counit ${}_\A\bfO\to\A$ defined by 
    \[\varphi\mapsto\sum_{\zeta\in X}\varphi(1_\zeta),\quad\forall\varphi\in{}_\A\bfO;\]
    \item antipode $S^*$, which is defined to be the dual of the antipode $S$ of ${}_\A\dot{\bfU}$.
\end{itemize}
For any commutative $\A$-algebra $R$ we define ${}_R\bfO\coloneqq{}_\A\bfO\otimes_\A R$. Then ${}_R\bfO$ is a Hopf $R$-algebra and it is commutative if $v\mapsto1$ in $R$. When $R=\ZZ$ and $v\mapsto1$, we simply write $\bfO\coloneqq{}_\ZZ\bfO$. Consider the following subsets of $\dot{\mathbf{B}}$ (see \cite[25.2.6]{Lusztig-IntroQuantumGroup}):
\[\dot{\mathbf{B}}^{\ge0}\coloneqq\Set{b^+1_{\lambda}| b\in {\mathbf{B}},\lambda\in X},\quad\dot{\mathbf{B}}^{\le0}\coloneqq\Set{b^-1_{\lambda}| b\in {\mathbf{B}}, \lambda\in X}\]
and the corresponding subset of the dual canonical basis $\dot{\bfB}^*$:
\[\dot{\mathbf{B}}^{\ge0,*}\coloneqq\{b^*,b\in\dot{\mathbf{B}}^{\ge0}\},\quad \dot{\mathbf{B}}^{\le0,*}\coloneqq\{b^*,b\in\dot{\mathbf{B}}^{\le0}\}.\]
Let ${}_R\bfO^{\ge0}$ (resp. ${}_R\bfO^{\le0}$) be the $R$-submodule of ${}_R\bfO$ spanned by $\dot{\mathbf{B}}^{\ge0,*}$ (resp. $\dot{\mathbf{B}}^{\le0,*}$). We can identify the intersection ${}_R\bfO^{\ge0}\cap{}_R\bfO^{\le0}$ with $R[X]$ as $R$-modules (for each $\chi\in X$, the element $e^\chi\in R[X]$ is identified with $1_\chi^*\in{}_R\bfO$). Then we have natural projections
\[\pi^{\ge0}\colon{}_R\bfO\epic{}_R\bfO^{\ge0},\quad\pi^{\le0}\colon{}_R\bfO\epic{}_R\bfO^{\le0},\quad\pi^0\colon{}_R\bfO\to R[X].\]
According to \cite[\S3.5]{LusztigGroupScheme}, there are natural Hopf $R$-algebra structures on ${}_R\bfO^{\ge0}$ and ${}_R\bfO^{\le0}$ such that the maps $\pi^{\ge0}$, $\pi^{\le0}$ and $\pi^0$ are surjective homomorphisms of Hopf $R$-algebras. 

\begin{defi}
    Similar to $\dot{\bfB}^*$, we define the dual canonical basis $\bfB^*\subset\mathrm{Hom_{\A\text{-}\mod}}({}_\A\bff,\A)$ and let ${}_\A\bff^\circ$ be the free $\A$-submodule of $\mathrm{Hom}_{\A\text{-}\mod}({}_\A\bff,\A)$ spanned by $\bfB^*$. For any commutative $\A$-algebra $R$ we denote ${}_R\bff^\circ\coloneqq{}_\A\bff^\circ\otimes_\A R$. As in \cite[\S3.8]{LusztigGroupScheme}, after the specialization $v\mapsto1$, the cocommutative Hopf algebra structure on ${}_\ZZ\bff$ induces a commutative Hopf algebra structure on ${}_\ZZ\bff^\circ$.  
\end{defi}

We have natural projections
\[\pi^{>0},\pi^{<0}\colon{}_R\bfO\epic{}_R\bff^\circ\]
defined by $\pi^{>0}((b^+1_\la)^*)=b^*$ (resp. $\pi^{<0}((b^-1_\la)^*)=b^*$) for all $b\in\bfB,\la\in X$, and all other dual canonical basis elements map to zero. Clearly $\pi^{>0}$ (resp. $\pi^{<0}$) factors through $\pi^{\ge0}$ (resp. $\pi^{\le0}$).\par
Now let $\k$ be a commutative ring, viewed as $\A$-algebra via $v\mapsto1$. The maps $\pi^{>0}$ and $\pi^{<0}$ are homomorphisms of $\k$-Hopf algebras by \eqref{eq:coproduct-formula} and induce $\k$-algebra isomorphisms
\[{}_\k\bfO^{\ge0}\cong{}_\k\bff^\circ\otimes_\k\k[X],\quad{}_\k\bfO^{\le0}\cong{}_\k\bff^\circ\otimes_\k\k[X].\]
By the above discussions $G\coloneqq\Spec({}_\k\bfO)$ is an affine flat group scheme over $\k$ and we have the following closed subgroup schemes of $G$: 
\begin{align*}
    T\coloneqq \Spec(\k[X]),\  B&\coloneqq\Spec({}_\k\bfO^{\ge0}),\ U\coloneqq\Spec({}_\k\bff^\circ),\\
    B^-&\coloneqq\Spec({}_\k\bfO^{\le0}),\ U^-\coloneqq\Spec({}_\k\bff^\circ)
\end{align*}
where $U$ (resp. $U^-$) embeds into $G$ via $\pi^{>0}$ (resp. $\pi^{<0}$). Moreover, we have $B=TU$ and $B^-=TU^-$.  

\begin{theo}\cite{LusztigGroupScheme}\label{theo:Chevalley-Lusztig group scheme}
    The group scheme $G\coloneqq\mathrm{Spec}({}_\k\bfO)$ is a split Chevalley group scheme over $\k$ associated to the based root datum $\Psi$. Moreover, $T$ is a maximal torus and $B$ is a Borel subgroup of $G$ that form part of the pinning of $G$.
\end{theo}
An important ingredient in Lusztig's proof is the following object:
\begin{defi}\label{defi:hatU}
    For any commutative $\A$-algebra $R$, the \emph{(quantized) distribution algebra} for $\Psi$ is the $R$-module ${}_R\hat{\bfU}\coloneqq\mathrm{Hom}_{\A\text{-}\mod}({}_\A\bfO,R)$. More concretely, ${}_R\hat{\bfU}$ consists of infinite formal sums $\sum_{a\in \dot{\bfB}}\xi_aa$ with $\xi_a\in R$. If $R=\QQ(v)$, we simply write $\hat{\bfU}\coloneqq{}_{\QQ(v)}\hat{\bfU}$. 
\end{defi}

The canonical basis $\dot{\bfB}$ gives rise to the linear topology on ${}_R\hat{\bfU}$ in which the $R$-submodules $\{\prod_{b\in\dot{\bfB}\setminus F}Rb\}$, where $F$ runs over all finite subsets of $\dot{\bfB}$, form a basis of open neighborhoods of $0$. Under the identification ${}_R\hat{\bfU}=\mathrm{Hom}_{\mathrm{Mod}_R}({}_R\bfO,R)$, this is the weak* topology, i.e. the coarsest topology on ${}_R\hat{\bfU}$ such that for any $f\in{}_R\bfO$, the $R$-linear map 
\[\mathrm{ev}_f\colon{}_R\hat{\bfU}\to R,\quad\xi\mapsto\xi(f)\]
is continuous, where $R$ is equipped with the discrete topology. \par
The Hopf algebra structure of ${}_R\bfO$ defines an $R$-algebra structure on ${}_R\hat{\bfU}$ and we have a natural embedding of $R$-algebras ${}_{R}\dot{\bfU}\hookrightarrow {}_R\hat{\bfU}$. The coproduct \eqref{eq:U-dot-coproduct} for ${}_R\dot{\bfU}$ extends uniquely to a continuous $R$-algebra homomorphism
\[\hat{\Delta}\colon {}_R\hat{\bfU}\to {}_R\hat{\bfU}\hat{\otimes}{}_R\hat{\bfU}\]
where the completed tensor product ${}_R\hat{\bfU}\hat{\otimes}{}_R\hat{\bfU}$ consists of (infinite) formal sums $\sum_{b,c\in\dot{\bfB}}r_{bc}(b\otimes c)$ where the coefficients $r_{bc}\in R$. The counit $\varepsilon$ and antipode $S$ extend uniquely to continuous homomorphisms $\hat{\varepsilon}\colon {}_R\hat{\bfU}\to R$ and $\hat{S}\colon {}_R\hat{\bfU}\to {}_R\hat{\bfU}^{\mathrm{op}}$.\par 
By a \emph{continuous ${}_R\hat{\bfU}$-module}, we mean a ${}_R\hat{\bfU}$-module $M$ such that the scalar multiplication map ${}_R\hat{\bfU}\times M\to M$ is continuous when $M$ is equipped with the discrete topology. It is easy to see that this condition is equivalent to requiring that the annihilator of any element $m\in M$ is open in ${}_R\hat{\bfU}$. Let ${}_R\hat{\bfU}\text{-}\mod$ be the category of \emph{continuous} ${}_R\hat{\bfU}$-modules and let ${}_R\hat{\bfU}\text{-}\mod^{ft}$ be the full subcategory consisting of modules that are finitely generated $R$-modules.
\begin{prop}\label{prop:uni-mod-equiv}
    The forgetful functor induces an equivalence of categories 
    \[{}_R\hat{\bfU}\text{-}\mod^{ft}\simeq{}_R\dot{\bfU}\text{-}\mod^{ft}.\]
\end{prop}
\begin{proof}
    Let $M$ be a continuous ${}_R\hat{\bfU}$-module that is finitely generated as an $R$-module. Choose $R$-module generators $m_1,\dotsc,m_s$. By continuity, the annihilator of each $m_j$ is open and so their finite intersection, which is the annihilator of $M$, is open. Consequently all but finitely many elements in $\dot{\bfB}$ act by $0$ on $M$. In particular, there is a finite subset $X_M\subset X$ such that $1_\zeta M=0$ for all $\zeta\notin X_M$. Therefore the finitely many idempotents $1_\la$ for $\la\in X_M$ give a decomposition $M=\oplus_{\la\in X_M}1_\la M$. Thus $M$ is a unital ${}_R\dot{\bfU}$-module.\par 
    Conversely let $M$ be a unital ${}_R\dot{\bfU}$-module that is finitely generated as an $R$-module. By \cite[1.10(a)]{LusztigGroupScheme}, the ${}_R\dot{\bfU}$-action on $M$ extends to a continuous ${}_R\hat{\bfU}$-action.
\end{proof}
Following \cite{LusztigGroupScheme} we define the following sub-algebras of ${}_R\hat{\bfU}$:
\begin{align}\label{eq:subalgebra-U-hat}
    {}_R\hat{\bfU}^{\ge0}&\coloneqq\{\sum_{\la\in X,b\in\bfB}\xi_{\la,b}b^+1_\la\mid\xi_{\la,b}\in R\},\quad{}_R\hat{\bfU}^{+}\coloneqq\{\sum_{b\in\bfB,\la\in X}\xi_bb^+1_\la\mid \xi_b\in R\},\\
    {}_R\hat{\bfU}^{\le0}&\coloneqq\{\sum_{\la\in X,b\in\bfB}\xi_{\la,b}b^-1_\la\mid\xi_{\la,b}\in R\},\quad{}_R\hat{\bfU}^{-}\coloneqq\{\sum_{b\in\bfB,\la\in X}\xi_bb^-1_\la\mid \xi_b\in R\},\\
    \quad{}_R\hat{\bfU}^0&\coloneqq{}_R\hat{\bfU}^{\ge0}\cap{}_R\hat{\bfU}^{\le0}=\{\sum_{\la\in X}\xi_\la 1_\la\mid\xi_\la\in R\}.
\end{align}

For a commutative $\A$-algebra $\k$ with $v\mapsto1$, the group of $\k$-valued points $G(\k)$ is identified with a subgroup of the multiplicative monoid of ${}_\k\hat{\bfU}$:
\begin{equation}\label{eq:be in G}
    G(\k)=\{\xi\in{}_\k\hat{\bfU}\mid\hat{\Delta}(\xi)=\xi\otimes\xi\text{ and }\hat{\varepsilon}(\xi)=1\}.
\end{equation} 
Under this identification, the inverse map on $G(\k)$ is given by $\hat{S}$ and we have a description of the following subgroups:
\[T(\k)=G(\k)\cap{}_\k\hat{\bfU}^0,\quad B(\k)=G(\k)\cap{}_\k\hat{\bfU}^{\ge0},\quad U(\k)=G(\k)\cap{}_\k\hat{\bfU}^{+},\]
\[B^-(\k)=G(\k)\cap{}_\k\hat{\bfU}^{\le0},\quad U^-(\k)=G(\k)\cap{}_\k\hat{\bfU}^{-}.\]

The following technical result will be used in \cite{CJ2} to prove a result about embedding reductive monoids into products of matrix monoids.

\begin{prop}\label{prop:unip-element-trivial}
    Let $\k$ be a commutative $\A$-algebra with $v\mapsto1$. Let $x\in U^-(\k)=G(\k)\cap{}_\k\hat{\bfU}^{-}$ be an element which we write in the form
    \[x=\sum_{b\in\bfB,\la\in X}x_b(b^-1_\la),\quad x_b\in\k.\]
    Let $\la_1,\dotsc,\la_N\in X^+\setminus\{0\}$ be elements that generate the abelian group $X$. Suppose that for any $i$ and any $b\in\bfB(\la_i)\setminus\{1\}$ we have $x_b=0$. Then we have $x=1$. In other words, $x_b=0$ for all $b\in\bfB\setminus\{1\}$. 
\end{prop}
\begin{proof}
    Since $x\in G(\k)$, we have $x_{1}=1$ and $\hat{\Delta}(x)=x\otimes  x$. Since
    \[x\otimes  x=\sum_{a,b\in\bfB}\sum_{\la,\mu\in X}x_ax_b(a^-1_\la\otimes  b^-1_\mu),\]
    we must have 
    \[\hat{\Delta}(x)=\sum_{c\in\bfB,\nu\in X}x_c\cdot\hat{\Delta}(c^-1_\nu)=\sum_{a,b,c\in\bfB}\sum_{\la,\mu,\nu\in X} x_c\hat{m}_{c^-1_\nu}^{a^-1_\la,b^-1_\mu}a^-1_\la\otimes  b^-1_\mu\]
    and then we deduce that for any $a,b\in\bfB$ and $\la,\mu\in X$,
    \[x_ax_b=\sum_{c\in\bfB}\sum_{\nu\in X}x_c\hat{m}_{c^-1_\nu}^{a^-1_\la,b^-1_\mu}.\]
    Let $\L\subset X^+$ be the submonoid generated by the elements $\la_1,\dotsc,\la_N$. By Lemma \ref{lem:B-union} we have $\bfB=\bigcup_{\la\in\L}\bfB(\la)$. We check that the condition in this Lemma is satisfied as follows. From the assumption that $\la_1,\dotsc,\la_N$ generate the abelian group $X$, each simple coroot has positive pairing with at least one of generators $\la_j$, so its pairing with sufficiently large multiples of $\sum_{j=1}^N\la_j$ will be greater than any fixed positive number.\par  
    We now prove by induction on the length of the expression of $\la\in\L$ on the generators $\la_1,\dotsc,\la_N$, that for any $b\in\bfB(\la)\setminus\{1\}$ we have $x_b=0$. This will imply that $x=1$. The case $\la=0$ is clear since $\bfB(0)=\{1\}$. Suppose the statement is true for $\la$ and fix any $1\le i\le N$. We take $\mu=\la_i$ in the equality above and let $a\in\bfB(\la)$, $b\in\bfB(\la_i)$. Then we have $\hat{m}_{c^-1_\nu}^{a^-1_\la,b^-1_{\la_i}}=0$ unless $\nu=\la+\la_i$ and $c\in\bfB(\la+\la_i)$. Consequently, we get that
    \[x_ax_b=\sum_{c\in\bfB(\la+\la_i)}x_c\hat{m}_{c^-1_{\la+\la_i}}^{a^-1_\la,b^-1_{\la_i}}.\]
    Consider the function $h$ from Lemma \ref{lem:h_c-identity} below. For any element $d\in\bfB(\la+\la_i)\setminus\{1\}$, we get
    \begin{equation*}
        \begin{split}
            &\sum_{a\in\bfB(\la),b\in\bfB(\la_i)}h_{d^-1_{\la+\la_i}}(a^-1_\la,b^-1_{\la_i})x_ax_b \\&=\sum_{c\in\bfB(\la+\la_i)}x_c\sum_{a\in\bfB(\la),b\in\bfB(\la_i)}h_{d^-1_{\la+\la_i}}(a^-1_\la,b^-1_{\la_i})\hat{m}_{c^-1_{\la+\la_i}}^{a^-1_\la,b^-1_{\la_i}}\\
            &=x_{d}.
        \end{split}
    \end{equation*}
    By assumption and induction hypothesis, we have $x_ax_b=0$ unless $a=b=1$. Moreover, since $d\ne1$ we get by Lemma \ref{lem:h_c-identity} below that
    \[x_d=h_{d^-1_{\la+\la_i}}(1_\la,1_{\la_i})=0.\]
    This finishes the inductive argument, and we conclude that $x_b=0$ for all $b\ne1$. 
\end{proof}

The following lemma contains special cases of results in \cite[\S1.15]{LusztigGroupScheme}, and will be used to prove Proposition \ref{prop:unip-element-trivial}.
\begin{lem}\label{lem:h_c-identity}
    Let $\la,\la'\in X^+$. For each $c\in\dot{\bfB}_{0,\la+\la'}$, there exists a function 
    \[h_c\colon\dot{\bfB}_{0,\la}\times\dot{\bfB}_{0,\la'}\to\A\] 
    satisfying the following properties:
    \begin{itemize}
        \item $h_c(1_\la,1_{\la'})=0$ if $c\ne1_{\la+\la'}$;
        \item For any $c,c'\in\dot{\bfB}_{0,\la+\la'}$ the following identity holds
        \[\sum_{a\in\dot{\bfB}_{0,\la},b\in\dot{\bfB}_{0,\la'}}h_c(a,b)\hat{m}_{c'}^{a,b}=\delta_{c,c'},\quad\forall c,c'\in\dot{\bfB}_{0,\la+\la'}.\]
    \end{itemize}
\end{lem}
\begin{proof}
Let $\tau\colon{}_\A\Lambda_{\lambda+\lambda'}\to{}_\A\Lambda_\lambda\otimes_\A {}_\A\Lambda_{\lambda'}$ be the unique ${}_\A\dot{\bfU}$-module homomorphism such that $\tau(\eta_{\lambda+\lambda'})=\eta_\la\otimes \eta_{\la'}$. Then we have the following commutative diagram:
\[\xymatrix{
{}_\A\hat{\bfU}\ar[r]^{\hat{\Delta}}\ar[d] & {}_\A\hat{\bfU}\hat{\otimes}_{\A}{}_\A\hat{\bfU}\ar[d] \\
{}_\A\La_{\la+\la'}\ar[r]^{\tau} & {}_\A\La_\la\otimes_{\A} {}_\A\La_{\la'}
}\]
where the left (resp. right) vertical map sends $u\in{}_\A\hat{\bfU}$ (resp. $u\in{}_\A\hat{\bfU}\hat{\otimes}_\A{}_\A\hat{\bfU}$) to $u\cdot\eta_{\la+\la'}$ (resp. $u\cdot(\eta_\la\otimes \eta_{\la'})$).\par 
By \cite[27.1.7]{Lusztig-IntroQuantumGroup}, $\tau$ maps an $\A$-basis of ${}_\A\La_{\la+\la'}$ into an $\A$-basis of ${}_\A\La_\la\otimes_{\A} {}_\A\La_{\la'}$. So it is a split injection and we can find an $\A$-linear map 
\[\psi\colon{}_\A\La_\la\otimes_{\A} {}_\A\La_{\la'}\to{}_\A\La_{\la+\la'}\]
such that $\psi\tau=1$. See also \cite[1.13(a)]{LusztigGroupScheme} for a more general result.\par  
For each $c\in\dot{\bfB}_{0,\la+\la'}$, there exists a function $h_c\colon\dot{\bfB}_{0,\la}\times\dot{\bfB}_{0,\la'}\to\A$ such that 
\[\psi(a\eta_\la\otimes  b\eta_{\la'})=\sum_{c\in\dot{\bfB}_{0,\la+\la'}}h_c(a,b)c\eta_{\la+\la'}.\]
Since $\psi(\eta_\la\otimes \eta_{\la'})=\psi\tau(\eta_{\la+\la'})=\eta_{\la+\la'}$, we see that for any $c\in\dot{\bfB}_{0,\la+\la'}$ with $c\ne1_{\la+\la'}$ we have $h_c(1_\la,1_{\la'})=0$. \par 
The commutative diagram above implies that
\[\tau(c\eta_{\la+\la'})=\sum_{a\in\dot{\bfB}_{0,\la},b\in\dot{\bfB}_{0,\la'}}\hat{m}_c^{ab}a\eta_{\la}\otimes  b\eta_{\la'}.\]
Therefore the identity $\psi\tau=1$ implies that 
\[\sum_{a\in\dot{\bfB}_{0,\la},b\in\dot{\bfB}_{0,\la'}}h_c(a,b)\hat{m}_{c'}^{a,b}=\delta_{c,c'},\quad\forall c,c'\in\dot{\bfB}_{0,\la+\la'}.\]
\end{proof}

\subsection{The canonical filtration}
Recall from \S\ref{subsubsec:based-root-data} that $X^+_\pos$ denotes the submonoid of $X$ generated by the dominant weight monoid $X^+$ and the positive root monoid $X_\pos=\NN^\Delta$.
\begin{defi}\label{def:O-le-lambda}
    Let $R$ be a commutative $\A$-algebra. For any $\la\in X^+_\pos$, let ${}_R\bfO_{\le\la}$ be the $R$-submodule of ${}_R\bfO$ spanned by elements in $\bigcup\limits_{\mu\in X^+,\mu\le\la}\dot{\bfB}[\mu]^*$. Similarly, let ${}_R\bfO_{<\la}$ be the $R$-submodule of ${}_R\bfO$ spanned by elements in $\bigcup\limits_{\mu\in X^+,\mu<\la}\dot{\bfB}[\mu]^*$.
    In the classical situation, i.e. when $R=\k$ is an $\A$-algebra with $v\mapsto1$, we also use the more familiar notation $\k[G]_{\le\la}={}_\k\bfO_{\le\la}$ and $\k[G]_{<\la}={}_\k\bfO_{<\la}$. 
\end{defi}

\begin{lem}\label{lem:filtration-O-multiplicative}
    For any commutative $\A$-algebra $R$ and any elements $\lambda,\mu\in X^+_\pos$ we have: 
    \begin{enumerate}
        \item[(i)] ${}_R\mathbf{O}_{\leq\lambda}\cdot{}_R\mathbf{O}_{\leq\mu}\subset{}_R\mathbf{O}_{\leq\lambda+\mu}$;
        \item[(ii)] If $\la,\mu\in X^+$ are dominant weights, then $1_\la^*\cdot 1_\mu^*=1_{\la+\mu}^*$.
    \end{enumerate}
    Consequently, we have an $R$-algebra embedding $R[X^+]\into{}_R\bfO$ sending $e^\la$ to $1_\la^*$. Here $R[X^+]$ is the monoid algebra of $X^+$ with $R$-basis $\{e^\la,\la\in X^+\}$. 
\end{lem}
\begin{proof}
    We may and do assume that $R=\A$.\par  
    (i) For any $a\in\dot{\bf{B}}$, we let $\lambda(a)$ denote the unique element $\lambda\in X^+$ such that $a\in\dot{\bf{B}}[\lambda]$. We need to show that for any $a,b,c\in\dot{\bf{B}}$ such that $\hat{m}_c^{ab}\ne0$, we have $\lambda(c)\le\lambda(a)+\lambda(b)$. Suppose on the contrary that there exists such a triple $a,b,c$ with $\hat{m}_c^{ab}\ne0$ but $\lambda(c)\nleq\lambda(a)+\lambda(b)$. We may assume that among all such triples with $c$ fixed, the pair $(\lambda(a),\lambda(b))$ is minimal. In other words, we assume that if $x,y\in\dot{\bf{B}}$ satisfy the following conditions:
    \begin{itemize}
        \item $\hat{m}_c^{xy}\ne0$ and $\lambda(c)\nleq \lambda(x)+\lambda(y)$,
        \item $\lambda(x)\le\lambda(a)$ and $\lambda(y)\le\lambda(b)$.
    \end{itemize}
    Then we have $\lambda(x)=\lambda(a)$ and $\lambda(y)=\lambda(b)$. \par 
    Since $\lambda(c)\nleq\lambda(a)+\lambda(b)$, by \cite[Lemma 4.21]{Lusztig-Infinity} we get that $\Delta(c)=\sum_{x,y\in\dot{\bf{B}}}\hat{m}_c^{xy}x\otimes  y$ acts by $0$ on $\Lambda_{\lambda(a)}\otimes \Lambda_{\lambda(b)}$. If a summand $x\otimes  y$ with $\hat{m}_c^{xy}\ne0$ does not act by $0$ on $\Lambda_{\lambda(a)}\otimes \Lambda_{\lambda(b)}$, then by \cite[\S29.1.3(d)]{Lusztig-IntroQuantumGroup} we have $\lambda(x)\le\lambda(a)$ and $\lambda(y)\le\lambda(b)$ and therefore by our assumption above, $\lambda(x)=\lambda(a)$ and $\lambda(y)=\lambda(b)$. This means that the sum
    \[\sum_{x\in\dot{\bf{B}}[\lambda(a)],y\in\dot{\bf{B}[\lambda(b)]}}\hat{m}_c^{xy}x\otimes  y\]
    acts by $0$ on $\Lambda_{\lambda(a)}\otimes \Lambda_{\lambda(b)}$.\par 
    From \cite[\S29.1.6]{Lusztig-IntroQuantumGroup} we see that under the natural map $\dot{\bf{U}}[\ge\lambda(a)]\to\mathrm{End}(\Lambda_{\lambda(a)})$, the set $\dot{\bf{B}}[\lambda(a)]$ maps to a basis (and similarly for $\lambda(b)$). Thus, the images of the nonzero summands above in $\mathrm{End}(\Lambda_{\lambda(a)}\otimes \Lambda_{\lambda(b)})$ are linearly independent and hence all vanish. Then $a\otimes  b$ would also act by $0$ on $\Lambda_{\lambda(a)}\otimes \Lambda_{\lambda(b)}$ since $\hat{m}_c^{ab}\ne0$ by assumption. This contradicts the fact that $a$ (resp. $b$) does not act by $0$ on $\Lambda_{\lambda(a)}$ (resp. $\Lambda_{\lambda(b)}$).\par 
    (ii) For any $c\in\dot{\bfB}$ with $\hat{m}_c^{1_\la,1_\mu}\ne0$, we must have $c\in 1_{\la+\mu}\dot\bfU 1_{\la+\mu}\cap\dot{\bfB}[\le\la+\mu]$ by (i) and the definition of coproduct for $\dot{\bfU}$. From Proposition \ref{prop:parametrize-B[la]} we deduce that the only possibility is $c=1_{\la+\mu}$ with $\hat{m}_{1_{\la+\mu}}^{1_\la,1_\mu}=1$. Therefore $1_\la^*1_\mu^*=1_{\la+\mu}^*$.
\end{proof}
As a consequence we can prove the following result, which will be used in the study of very flat reductive monoids in \cite{CJ2}.
\begin{lem}\label{lem:B-dot-X-ab}
    Let $\chi\in X_{ab}$. In the notation of Definition \ref{def:B-dot} we have
    \[1_\chi^*\cdot(b_1\diamondsuit_{\zeta-\chi}b_2)^*=(b_1\diamondsuit_{\zeta}b_2)^*,\quad\forall b_1,b_2\in\bfB,\forall\zeta\in X.\]
    In particular $1_0^*$ is the multiplicative unit. Moreover, for any $\la\in X^+$ 
    the map $b^*\mapsto 1_\chi^*b^*$ defines a canonical bijection:
    \[\dot{\bfB}[\la]^*\xrightarrow{\sim}\dot{\bfB}[\la+\chi]^*.\]
\end{lem}
\begin{proof}
    Take any $b_1,b_2\in\bfB$ and $\zeta,\eta\in X$ such that $b_1^+b_2^-1_\zeta\in 1_\eta\dot{\bfU}1_\zeta$. We write
    \[r(b_1)=\sum_{j=1}^{m_1}x_j\otimes x'_j,\quad \bar{r}(b_2)=\sum_{k=1}^{m_2}{}_ky\otimes {}_ky'\]
    where $m_1,m_2\in\ZZ_{\ge1}$, $x_j,x_j',{}_ky,{}_ky'$ are all homogeneous elements in $\bff$, and we require that $x_j=1$ (resp. ${}_ky=1$) if and only if $j=1$ (resp. $k=1$). Then we have $x_1'=b_1$, ${}_1y'=b_2$ and \[|x_j|+|x_j'|=|b_1|,\quad\forall j;\quad |{}_ky|+|{}_ky'|=|b_2|,\quad\forall k.\]
    From \eqref{eq:coproduct-formula} we get
    \[\Delta(b_1^+)=\sum_{j=1}^{m_1}\sum_{\la_1\in X}v^{|x_j'|\circ\la_1}x_j^+1_{\la_1}\otimes (x_j')^+,\]
    \[\Delta(b_2^-1_\zeta)=\sum_{k=1}^{m_2}\sum_{\la_2\in X}v^{-|{}_ky|\circ\la_2}1_{\zeta-|b_2|-\la_2}({}_ky)^-\otimes 1_{\la_2}({}_ky')^-.\]
    Combined together we get
    \[\Delta(b_1^+b_2^-1_\zeta)=\sum_{j=1}^{m_1}\sum_{k=1}^{m_2}\sum_{\substack{\la_1,\la_2\in X\\ \la_1+\la_2=\zeta-|b_2|}}v^{|x_j'|\circ\la_1-|{}_ky|\circ\la_2}x_j^+1_{\la_1}({}_ky)^-\otimes(x_j')^+1_{\la_2}({}_ky')^-.\]
    The only term in this sum that lie in $1_\chi\dot{\bfU}1_\chi\otimes 1_{\eta-\chi}\dot{\bfU}1_{\zeta-\chi}$ is $1_\chi\otimes b_1^+b_2^-1_{\zeta-\chi}$, for which $j=k=1$ and $\la_1=\chi\in X_{ab}$.\par
    Now for the element $b_1\diamondsuit_\zeta b_2\in\dot{\bfB}$ we write 
    \[\Delta(b_1\diamondsuit_\zeta b_2)=\sum_{b\in\dot{\bfB}}b\otimes\varphi(b),\quad\varphi(b)\in\hat{\bfU}.\]
    From the discussion above we deduce that 
    \[\varphi(1_\chi)=T_\chi(b_1\diamondsuit_\zeta b_2),\quad \chi\in X_{ab}. \]
    where $T_\chi$ is the $\QQ(v)$-linear automorphism on $\dot{\bfU}$ defined by 
    \[T_\chi(c_1^+c_2^-1_\xi)=c_1^+c_2^-1_{\xi-\chi},\quad\forall c_1,c_2\in\bfB, \xi\in X.\]
    We note that $T_\chi$ restricts to bijections $1_\eta\dot{\bfU}1_\zeta\to1_{\eta-\chi}\dot{\bfU}1_{\zeta-\chi}$ for all $\eta,\zeta\in X$ and moreover $T_\chi(u_1u_2)=T_\chi(u_1)T_\chi(u_2)$ for all $u_1,u_2\in\dot{\bfU}$. One verifies from the definitions that $T_\chi$ commutes with the involution $\rho$ and preserves the bilinear form in Proposition \ref{prop:bilinear pairing characterizing dotB}. Consequently we have $T_\chi(\dot{\bfB})\subset\pm\dot{\bfB}$ and by looking at the leading terms we deduce that  
    \[\varphi(1_\chi)=T_\chi(b_1\diamondsuit_\zeta b_2)=b_1\diamondsuit_{\zeta-\chi} b_2,\quad\forall b_1,b_2\in\bfB, \zeta\in X.\]
    Then by the definition of the product on ${}_\A\bfO$ we conclude that
    \[1_\chi^*\cdot(b_1\diamondsuit_{\zeta-\chi}b_2)^*=(b_1\diamondsuit_{\zeta}b_2)^*,\quad\forall b_1,b_2\in\bfB,\forall\zeta\in X.\]
    In particular we have $1_\chi^*\cdot b^*\in\dot{\bfB}^*$ and $1_0^*\cdot b^*=b^*$ for all $b\in\dot{\bfB}$, so $1_0^*$ is the identity in $\bfO$. If $b\in\dot{\bfB}[\la]$, then we have $1_\chi^*\cdot b^*\in\dot{\bfB}^*[\le\la+\chi]$ by Lemma \ref{lem:filtration-O-multiplicative}. Moreover since $1_\chi^*1_{-\chi}^*=1_0^*$ is the identity by Lemma \ref{lem:filtration-O-multiplicative}, we must have $1_\chi^*\cdot b^*\in\dot{\bfB}[\la+\chi]^*$.
\end{proof}

\subsection{Some basic functorial properties}\label{subsec:functorial-properties}
In this subsection we prove some basic functorial properties for the canonical bases. The results will be needed in \cite{CJ2} during the study of the decomposition of reductive monoids.\par 
Suppose we have a decomposition of the simple roots $\Delta=\Delta_1\sqcup\Delta_2$ such that $\langle\alpha_i^\vee,\alpha_j\rangle=0$ for any $i\in\Delta_1, j\in\Delta_2$. For $k=1,2$, let $\bff^{(k)}$ (resp. ${}_\A\bff^{(k)}$) be the $\QQ(v)$-algebra (resp. $\A$-algebra) defined in the same way as $\bff$ (resp. ${}_\A\bff$) with $\Delta$ replaced by $\Delta_k$. Let $\bfB^{(k)}$ be the canonical basis of $\bff^{(k)}$ and let $(\cdot,\cdot)_k\colon\bff^{(k)}\times\bff^{(k)}\to\QQ(v)$ be the bilinear form that characterizes $\pm\bfB^{(k)}$ as \eqref{eq:pairing characterizing B} above. 
\begin{prop}\label{prop:decomposition f=f1f2}
    We have a canonical isomorphism of $\QQ(v)$-algebras $\bff\cong\bff^{(1)}\otimes_{\QQ(v)}\bff^{(2)}$ and $\A$-algebras ${}_\A\bff\cong{}_\A \bff^{(1)}\otimes_\A{}_\A\bff^{(2)}$, under which $\bfB$ is mapped bijectively onto $\bfB^{(1)}\otimes\bfB^{(2)}$ and the bilinear form factorizes as $(\cdot,\cdot)=(\cdot,\cdot)_1\otimes(\cdot,\cdot)_2$. In other words, we have 
    \begin{itemize}
        \item any element $b\in\bfB$ can be written uniquely as $b=b_1\otimes b_2$ where $b_1\in\bfB^{(1)}$ and $b_2\in\bfB^{(2)}$,
        \item for any elements $x,x'\in\bff^{(1)}$ and $y,y'\in\bff^{(2)}$, we have $(xy,x'y')=(x,x')_1(y,y')_2$.
    \end{itemize}
\end{prop}
\begin{proof}
    For $k=1,2$, the natural embeddings $\bff^{(k)}\into\bff$ are compatible with the coproducts. Hence the restrictions of the pairing $(\cdot,\cdot)$ to $\bff^{(k)}$ coincide with $(\cdot,\cdot)_k$ by their defining properties in \cite[Proposition 1.2.3]{Lusztig-IntroQuantumGroup}. For any $i\in\Delta_1,j\in\Delta_2$ and $p,p'\in\NN$, the quantum Serre relation \eqref{eq:quantum Serre relations} becomes $\theta_i^{(p)}\theta_j^{(p')}=\theta_j^{(p')}\theta_i^{(p)}$. Thus the canonical embeddings $\bff^{(k)}\into\bff$ (for $k=1,2$) together with the multiplication on $\bff$ induce a $\QQ(v)$-algebra isomorphism $\bff^{(1)}\otimes_{\QQ(v)}\bff^{(2)}\cong\bff$. \par 
    For each $i\in\Delta$, recall from \cite[1.2.13]{Lusztig-IntroQuantumGroup} the $\QQ(v)$-linear map ${}_ir\colon\bff\to\bff$ satisfying the following properties:
    \begin{enumerate}
        \item[(i)] ${}_ir(1)=0$,
        \item[(ii)] ${}_ir(\theta_j)=\delta_{ij}$ for all $j\in\Delta$,
        \item[(iii)] ${}_ir(xy)={}_ir(x)\cdot y+v^{|x|\cdot i}x\cdot{}_ir(y)$ for all homogeneous $x,y\in\bff$,
        \item[(iv)] $(\theta_i y,x)=(\theta_i,\theta_i)(y,{}_ir(x))$ for all $x,y\in\bff$.
    \end{enumerate}
    Properties (i),(ii),(iii) uniquely characterize the map ${}_ir$. So if $i\in\Delta_k$ ($k=1,2$), the restriction of ${}_ir$ to $\bff^{(k)}$ is the similar map defined for $\bff^{(k)}$. Moreover, (iv) is a consequence of (i),(ii),(iii) and if $i\in\Delta_k$ ($k=1,2$), a similar identity holds for all $x,y\in\bff^{(k)}$ with $(\cdot,\cdot)$ replaced by $(\cdot,\cdot)_k$. \par
    Next we show the identity
    \begin{equation}\label{eq:pairing-factorize}
        (xy,x'y')=(x,x')_1(y,y')_2,\quad\forall x,x'\in\bff^{(1)}, y,y'\in\bff^{(2)}.
    \end{equation}
    Clearly it suffices to show this under the assumption that $x,x',y,y'$ are all homogeneous. We use induction on $\mathrm{tr}(|x|)\in\NN$. Since homogeneous elements of different degrees are orthogonal under the pairings, we have $(y,x'y')=0$ whenever $|x'|\ne0$ and \eqref{eq:pairing-factorize} holds if $x=1$. Suppose \eqref{eq:pairing-factorize} has been proved for some homogeneous $x\in\bff^{(1)}$ and arbitrary $y,x',y'$. Take any $i\in\Delta_1$. By (ii) and (iii) we get that ${}_ir(x'y')={}_ir(x')y'$ for all $x'\in\bff^{(1)},y'\in\bff^{(2)}$. Then by (iv) and the induction hypothesis we get
    \begin{align*}
        (\theta_i xy,x'y')&=(\theta_i,\theta_i)(xy,{}_ir(x')y')=(\theta_i,\theta_i)_1(x,{}_ir(x'))_1(y,y')_2\\
        &=(\theta_ix,x')_1(y,y')_2.
    \end{align*}
    This finishes the proof of \eqref{eq:pairing-factorize}. By \cite[Theorem 14.2.3]{Lusztig-IntroQuantumGroup}, we deduce that $b_1b_2\in\pm\bfB$ for any $b_1\in\bfB^{(1)}$ and $b_2\in\bfB^{(2)}$. It remains to remove the sign ambiguities.
    \begin{claim}
        For any $b_2\in\bfB\cap\bfB^{(2)}$ and $b_1\in\bfB^{(1)}$, we have $b_1b_2\in\bfB$.
    \end{claim}
    \begin{proof}[Proof of Claim]
        Recall from \cite[Theorem 14.3.2]{Lusztig-IntroQuantumGroup} that for any $i\in\Delta$ and $n\in\NN$, there is a unique bijection $\pi_{i,n}\colon\pm\bfB_{i;0}\to\pm\bfB_{i;n}$ such that for any $b\in\pm\bfB_{i;0}$ we have $\pi_{i,n}(b)-\theta_i^{(n)}b\in\theta_i^{n+1}\bff$. If $i\in\Delta_k$ ($k=1,2$), then we have similar maps $\pi_{i,n}^{(k)}\colon\pm\bfB_{i;0}^{(k)}\to\pm\bfB_{i;n}^{(k)}$ defined for $\bff^{(k)}$. Since $\theta_i^{(n)}\bff\cap\bff^{(k)}=\theta_i^{(n)}\bff^{(k)}$, we see that the map $\pi_{i,n}^{(k)}$ coincides with the restriction of $\pi_{i,n}$. 
        Also recall from \cite[14.4.2]{Lusztig-IntroQuantumGroup} that for any $\nu=\sum_{i\in \Delta}\nu_i i\in\NN^\Delta$ we have $\bfB_{\nu}=\bigcup_{i\in\Delta,0<n\le\nu_i}\pi_{i,n}(\bfB_{\nu-ni}\cap(\pm\bfB_{i;0}))$. A similar description holds for $\bfB^{(k)}$ ($k=1,2$).\par 
        We prove the claim by induction on $|b_1|\in\NN^{\Delta_1}$. If $b_1=1$ the claim is obvious. Now assume $|b_1|\ne0$. There exists $i\in\Delta_1$, $n\in\ZZ_{>0}$ and $b_1'\in\bfB^{(1)}_{|b_1|-ni}\cap(\pm\bfB^{(1)}_{i;0})$ such that $\pi_{i,n}^{(1)}(b_1')=b_1$. So we have $b_1'\notin\theta_i\bff$ and $b_1-\theta_i^{(n)}b_1'\in\theta_i^{n+1}\bff^{(1)}$. By induction hypothesis we have $b_1'b_2\in\bfB$. Clearly we also have $b_1'b_2\notin\theta_i\bff$ and hence $b_1'b_2\in\bfB_{|b_1|+|b_2|-ni}\cap(\pm\bfB_{i;0})$. Moreover, 
        \[b_1b_2-\theta_i^{(n)}b_1'b_2=(b_1-\theta_i^{(n)}b_1')b_2\in\theta_i^{n+1}\bff\]
        and therefore $b_1b_2=\pi_{i,n}(b_1'b_2)\in\bfB_{|b_1|+|b_2|}$. 
    \end{proof}
    Applying the claim to $b_2=1\in\bfB\cap\bfB^{(2)}$ we get that $\bfB^{(1)}\subset\bfB$. Similarly one shows that $\bfB^{(2)}\subset\bfB$. Then we apply the claim to any $b_2\in\bfB^{(2)}\subset\bfB$ and get a canonical embedding $\bfB^{(1)}\times\bfB^{(2)}\into\bfB$. This must be a bijection since both sides form a $\QQ(v)$-basis for $\bff$. 
\end{proof}

Let $\ZZ^{\Delta_2,\perp}\coloneqq\{\la\in X\mid\langle\alpha_j^\vee,\la\rangle=0,\forall j\in\Delta_2\}$. Then by assumption we have $\ZZ^{\Delta_1}\subset\ZZ^{\Delta_2,\perp}$. 
Let $X_1\subset X$ be a sub-lattice satisfying $\ZZ^{\Delta_1}\subset X_1\subset\ZZ^{\Delta_2,\perp}$ and let $Y_1\coloneqq\mathrm{Hom}_\ZZ(X_1,\ZZ)$. Then $\Delta_1,X_1,Y_1$ determine a based root datum $\Psi_1$, from which we define the modified quantized enveloping algebra $\dot{\bfU}_{\Psi_1}$ with its canonical basis $\dot{\bfB}_{\Psi_1}$. The positive/negative part of $\dot{\bfU}_{\Psi_1}$ is the subalgebra $\bff^{(1)}\subset\bff$ associated to $\Delta_1$ as before. From the defining relations we see that there is a natural embedding of (non-unital) $\QQ(v)$-algebras 
\begin{equation}\label{eq:morphism iota}
    \iota\colon\dot{\bfU}_{\Psi_1}\into\dot{\bfU}.
\end{equation} 
Note however that in general $\iota$ does not preserve the comultiplications. On the other hand, there is a unique projection $p\colon\dot{\bfU}\to\dot{\bfU}_{\Psi_1}$ such that $p\circ\iota$ is the identity on $\dot{\bfU}_{\Psi_1}$ and $p$ sends any $\theta_j^{\pm}1_{\chi}$ (with $j\in\Delta_2$ and $\chi\in X$) and any $1_\zeta$ (with $\zeta\notin X_1$) to $0$.
\begin{lem}\label{lem:map p}
    The map $p$ is a homomorphism of non-unital $\QQ(v)$-algebras that is compatible with the coproduct $\Delta_{\Psi_1}$ on $\dot{\bfU}_{\Psi_1}$ and $\Delta$ on $\dot{\bfU}$. In other words, for any $x\in\dot{\bfU}$ we have $(p\otimes p)(\Delta(x))=\Delta_{\Psi_1}(p(x))$. 
\end{lem}
\begin{proof}
    The natural projection $\bff\to\bff^{(1)}$ is clearly compatible with coproducts. We note that elements of the form $x^-1_\zeta\sigma(y)^+$ for $x,y\in\bfB$ and $\zeta\in X$ form a $\QQ(v)$-basis of $\dot{\bfU}$. Then the statement follows easily from \cite[Lemma 3.10]{Lusztig-Infinity}, which indicates that the comultiplication on $\dot{\bfU}$ (resp. $\dot{\bfU}_{\Psi_1}$) is uniquely determined by that on $\bff$ (resp. $\bff^{(1)}$). 
\end{proof}

\begin{prop}\label{prop:iota}
    For each $\la\in X_1^+$, the embedding $\iota$ maps $\dot{\bfB}_{\Psi_1}[\la]$ bijectively onto $\dot{\bfB}[\la]$.
\end{prop}
\begin{proof}
    It is clear that the embedding $\iota$ respects the orthogonal decomposition \eqref{eq:U-dot-decomposition} for $\dot{\bfU}_{\Psi_1}$ and $\dot{\bfU}$. Also one easily checks that the involutions $\sigma,\omega,\rho$ for $\dot{\bfU}$ restrict along $\iota$ to similar involutions defined for $\dot{\bfU}_{\Psi_1}$. Then one sees that the pairing $(\cdot,\cdot)$ on $\dot{\bfU}$ restricts along $\iota$ to the similarly defined pairing on $\dot{\bfU}_{\Psi_1}$. Thus we have $\iota(\dot{\bfB}_{\Psi_1})\subset\pm\dot{\bfB}$ by Proposition \ref{prop:bilinear pairing characterizing dotB}. Since the canonical base for $\bff^{(1)}$ is a subset of $\bfB$ by Proposition \ref{prop:decomposition f=f1f2}, we get that $\iota(\dot{\bfB}_{\Psi_1})\subset\dot{\bfB}$ by comparing the leading term.\par 
    Next we show the inclusion
    \[\iota(\dot{\bfU}_{\Psi_1}[\ge\la])\subset\dot{\bfU}[\ge\la],\qquad\forall\la\in X_1^+.\]
    For this we may assume that $\Psi$ is simply connected (so $\Psi_1$ is also simply connected) by \cite[23.2.5, 25.2.2]{Lusztig-IntroQuantumGroup}. Let $\{\omega_i\in\QQ^\Delta,i\in\Delta\}$ be the fundamental weights (i.e. $\langle\alpha_i^\vee,\omega_j\rangle=\delta_{ij}$ for all $i,j\in\Delta$). Then we have $X_1=\bigoplus_{i\in \Delta_1}\ZZ\omega_i$ and $X=X_1\oplus X_2$ where $X_2\coloneqq\bigoplus_{j\in\Delta_2}\ZZ\omega_j$. Let $\Psi_2$ be the root datum determined by $(X_2, \Delta_2)$. Then we have $\dot{\bfU}=\dot{\bfU}_{\Psi_1}\otimes\dot{\bfU}_{\Psi_2}$ and moreover
    \[X^+=X_1^+\oplus X_2^+,\quad X_1^+=\bigoplus_{i\in\Delta_1}\NN\omega_i,\quad X_2^+=\bigoplus_{i\in \Delta_2}\NN\omega_i.\]
    Thus any $\la'\in X^+$ is written uniquely as $\la'=\la_1'+\la_2'$ with $\la_1'\in X_1^+$, $\la_2'\in X_2^+$, and we have $\La_{\la'}=\La_{\la_1'}\otimes\La_{\la_2'}$ where (for $k=1,2$) $\La_{\la_k'}$ is the simple highest weight module for $\dot{\bfU}_{\Psi_k}$ with highest weight $\la_k'$. Then for any $x\in\dot{\bfU}_{\Psi_1}[\ge\la]$ such that $\iota(x)$ acts by a nonzero endomorphism on $\La_{\la'}$, we get that $x$ acts by a nonzero endomorphism on $\La_{\la_1'}$ and the zero weight space of $\La_{\la_2'}$ is nontrivial. This implies that $\la'\ge\la$ and hence $\iota(x)\in\dot{\bfU}[\ge\la]$ by Lemma \ref{lem:characterization subset Blambda}. The inclusion above follows. \par 
    Consequently for any $b\in\dot{\bfB}_{\Psi_1}[\la]$, we have $\iota(b)\in\dot{\bfB}[\mu]$ with $\mu\in X^+$ and $\mu\ge\la$. Since the irreducible module $\La_\la$ of highest weight $\la$ for $\dot{\bfU}_{\Psi_1}$ and $\dot{\bfU}$ are canonically identified, the element $\iota(b)$ acts by a nonzero endomorphism on it and therefore $\la\ge\mu$ by Lemma \ref{lem:characterization subset Blambda}. So we must have $\la=\mu$. On the other hand, since the sets $\dot{\bfB}_{\Psi_1}[\la]$ and $\dot{\bfB}[\la]$ have the same cardinality, which is the square of the $\QQ(v)$-dimension of $\La_\la$, the embedding $\iota$ induces a bijection between the sets $\dot{\bfB}_{\Psi_1}[\la]$ and $\dot{\bfB}[\la]$. 
\end{proof}

Now we consider two based root data with the same set of simple roots and Cartan matrix: 
\[\Psi=(\Delta,X,Y,...),\quad\Psi'=(\Delta,X',Y',...)\]
and let $G,G'$ be the corresponding Chevalley group schemes (over a common fixed base scheme $S$). Let $\dot{\bfU}'$ (resp. $\dot{\bfB}'$) be the modified quantized enveloping algebra associated to $\Psi'$ (resp. its canonical basis). 

\begin{lem}\label{lem:B-dot-functorial}
    Let $\varphi\colon G'\to G$ be a homomorphism of group schemes that is compatible with the pinnings and determines a homomorphism of root data $f\colon X\to X'$ (which restricts to the identity on the root lattice $\ZZ^\Delta$). Let $\varphi^*\colon\O_S[G]\to\O_S[G']$ be the induced homomorphism on coordinate rings. Then we have:
    \begin{enumerate}
        \item For any $b_1,b_2\in\bfB$ and any $\zeta\in X$ we have 
        \[\varphi^*((b_1\diamondsuit_\zeta b_2)^*)=(b_1\diamondsuit_{f(\zeta)}b_2)^*.\]
        \item For any $\la\in X^+$ and $b_1,b_2\in\bfB(\la)$, in the notation of Proposition \ref{prop:parametrize-B[la]} we have
        \[\varphi^*(\beta_\la(b_1,b_2)^*)=\beta_{f(\la)}(b_1,b_2)^*.\]
        In particular, $\varphi^*$ induces a canonical bijection $\dot{\bfB}[\la]^*\xrightarrow{\sim}\dot{\bfB}'[f(\la)]^*$.
    \end{enumerate}
\end{lem}
\begin{proof}
    (1) We identify the positive/negative parts of the quantized enveloping algebras for $G$ and $G'$. The natural homomorphism $\varphi_*\colon\hat{\bfU}_{G'}\to\hat{\bfU}_G$ sends $1_{\zeta'}$, where $\zeta'\in X'$, to the element 
    \[1_{f^{-1}(\zeta')}\coloneqq\sum_{\zeta\in f^{-1}(\zeta')}1_\zeta.\]
    (We adopt the convention that a sum over an empty set is defined to be $0$.) More generally, from the definition of canonical basis elements recalled in Definition \ref{def:B-dot} we deduce that
    \[\varphi_*(b_1\diamondsuit_{\zeta'} b_2)=\sum_{\zeta\in f^{-1}(\zeta')}b_1\diamondsuit_\zeta b_2\]
    and this immediately implies the result on the dual canonical basis.\par 
    (2) Let $\la'=f(\la)$. By (1) we have $\varphi^*(\beta_\la(b_1,b_2)^*)\in\dot{\bfB}'^*$. So there exists $\mu'\in X'^+$ and $c_1,c_2\in\bfB(\mu')$ such that $\varphi^*(\beta_\la(b_1,b_2)^*)=\beta_{\mu'}(c_1,c_2)^*$. Since $c_1^-1_{\mu'}\sigma(c_2)^+-\beta_{\mu'}(c_1,c_2)\in\dot{\bfU}'[>\mu']$, we have
    \[\beta_\la(b_1,b_2)^*(\varphi_*(c_1^-1_{\mu'}\sigma(c_2)^+))=1.\]
    Since $$\varphi_*(c_1^-1_{\mu'}\sigma(c_2)^+)=\sum_{\mu\in f^{-1}(\mu')}c_1^-1_{\mu}\sigma(c_2)^+$$ and each summand satisfies $c_1^-1_{\mu}\sigma(c_2)^+-\beta_\mu(c_1,c_2)\in\dot{\bfU}[>\mu]$,  
    there exists $\mu\in X^+$ with $f(\mu)=\mu'$ such that $\la\ge\mu$ and thus $\la'=f(\la)\ge f(\mu)=\mu'$.
    On the other hand, we have
    \[\varphi_*(b_1^-1_{\la'}\sigma(b_2)^+)=\sum_{\xi\in f^{-1}(\la')}b_1^-1_{\xi}\sigma(b_2)^+\]
    which implies that
    \[\beta_{\mu'}(c_1,c_2)^*(b_1^-1_{\la'}\sigma(b_2)^+)=\varphi^*(\beta_\la(b_1,b_2)^*) (b_1^-1_{\la'}\sigma(b_2)^+)=1\]
    and hence $\mu'\ge\la'$ by similar reasoning. So we conclude that $\mu'=\la'$ and then the equality above implies that $c_1=b_1$, $c_2=b_2$ and we are done. 
\end{proof}

\section{Finite generation of quantized coordinate rings}\label{sec:finite-coordinate-ring}
In this section our goal is to show that certain subalgebras of the quantized coordinate ring are finitely generated and become normal after the specialization $v\mapsto1$. These include the quantized coordinate rings of reductive monoids that we will construct in \cite{CJ2}.\par 
We let $R$ be a commutative $\A$-algebra (in which the image of $v$ is not necessarily $1$). Fix a based root datum $\Psi$ as in \S\ref{subsubsec:based-root-data} and keep the notations from \S\ref{sec:quantum-canonical-bases} and \S\ref{sec:Lusztig-group-scheme}. 

\subsection{The associated graded algebra}
\begin{defi}\label{defi:filtration O}
    For any subset $P\subset X^+$, let ${}_R\bfO(P)$ be the $R$-submodule of ${}_R\bfO$ spanned by the elements
    \[\{b^*\mid b\in\bigcup_{\la\in P}\dot{\bfB}[\la]\}.\] 
\end{defi}

\begin{prop}\label{prop:downward-closed}
    Let $P\subset X^+$ be a subset. The following are equivalent:
    \begin{enumerate}
        \item[(i)] ${}_R\bfO(P)$ is a sub-coalgebra of ${}_R\bfO$;
        \item[(ii)] ${}_R\dot{\bfU}[X^+\setminus P]$ is a two-sided ideal of ${}_R\dot{\bfU}$;
        \item[(iii)] $P$ is downward closed in $X^+$, i.e. for all $\la,\mu\in X^+$ with $\mu\le\la$, if $\la\in P$, then $\mu\in P$.
    \end{enumerate}
\end{prop}
\begin{proof}
    (i)$\Rightarrow$(ii): Suppose that ${}_R\bfO(P)$ is a sub-coalgebra. This means that \[\Delta_\bfO({}_R\bfO(P))\subset{}_R\bfO(P)\otimes_{R}{}_R\bfO(P)\] 
    where $\Delta_\bfO$ is the comultiplication for ${}_R\bfO$. Since ${}_R\dot{\bfU}[X^+\setminus P]$ is the orthogonal complement of ${}_R\bfO(P)$ in ${}_R\dot{\bfU}$ (under the canonical pairing between ${}_R\dot{\bfU}$ and ${}_R\bfO$), we deduce from the inclusion above that ${}_R\dot{\bfU}[X^+\setminus P]$ is a two-sided ideal.\par
    (ii)$\Rightarrow$(iii): Let $\la,\mu\in X^+$ be such that $\mu\le\la$ and $\la\in P$. Then $\mu$ is a weight of $\La_\la$. Since the image of $\dot{\bfB}[\la]$ under the natural map $\dot{\bfU}\to\mathrm{End}_{\QQ(v)}(\La_\la)$ forms a $\QQ(v)$-basis and each element of $\dot{\bfB}$ lies in a weight component $1_\zeta\dot{\bfU}1_{\zeta'}$ for some $\zeta,\zeta'\in X$, there exists an element $c\in\dot{\bfB}[\la]\cap 1_\mu\dot{\bfU}1_\mu$. If $\mu\notin P$, then $1_\mu\in{}_R\dot{\bfU}[X^+\setminus P]$ and hence $c=c1_\mu=1_\mu c\in{}_R\dot{\bfU}[X^+\setminus P]$ by condition (ii). But this contradicts the assumption that $\la\in P$. Therefore we must have $\mu\in P$. \par 
    (iii)$\Rightarrow$(i): Let $\la\in P$ and $c\in\dot{\bfB}[\la]$. For any $\la',\la''\in X^+$, $a\in\dot{\bfB}[\la']$ and $b\in\dot{\bfB}[\la'']$ such that $m_{ab}^c\ne0$, we have $\la\ge\la'$ and $\la\ge\la''$ since $ab\in{}_R\dot{\bfU}[\ge\la']\cap{}_R\dot{\bfU}[\ge\la'']$. By (iii), this implies that $\la'\in P$ and $\la''\in P$. Therefore $\Delta_\bfO(c^*)\in{}_R\bfO(P)\otimes_R{}_R\bfO(P)$. So ${}_R\bfO(P)$ is a sub-coalgebra. 
\end{proof}
\begin{ex}
    Let $\lambda\in X^+_\pos$ (see \S\ref{subsubsec:based-root-data}) and define
    \[X^+_{\le\la}\coloneqq\{\mu\in X^+\mid\mu\le\la\},\quad X^+_{<\la}\coloneqq\{\mu\in X^+\mid\mu<\la\}.\]
    Then we get the $R$-module ${}_R\bfO(X^+_{\le\la})={}_R\mathbf{O}_{\leq\lambda}$ and ${}_R\bfO(X^+_{<\la})={}_R\mathbf{O}_{<\lambda}$ from Definition \ref{def:O-le-lambda}. It is clear that the sets $X^+_{\le\la}$ and $X^+_{<\la}$ are both downward closed in $X^+$. So we see that ${}_R\bfO(X^+_{\le\la})$ and ${}_R\bfO(X^+_{<\la})$ are sub-coalgebras of ${}_R\bfO$. 
\end{ex}

\begin{defi}\label{def:grO}
    For any subset $P\subset X^+$, we define the associated graded module for ${}_R\bfO(P)$ to be 
    \[\gr({}_R\bfO(P))\coloneqq\bigoplus_{\la\in P}{}_R\bfO_{\le\la}/{}_R\bfO_{<\la}.\]
    If $P=X^+$, we simply denote $\gr({}_R\bfO)=\gr({}_R\bfO(X^+))$. When $P$ is a submonoid of $X^+$, then $\gr({}_R\bfO(P))$ admits a natural $P$-graded $R$-algebra structure by Lemma \ref{lem:filtration-O-multiplicative} and it is a subalgebra of $\mathrm{gr}({}_R\bfO)$. \par
    Moreover, since ${}_R\bfO(X^+_{\le\la})$ and ${}_R\bfO(X^+_{<\la})$ are sub-coalgebras of ${}_R\bfO$, the ring $\gr({}_R\bfO)$ inherits a coalgebra structure from ${}_R\bfO$. Note however that in general, the counit on ${}_R\bfO$ does not induce a counit on $\gr({}_R\bfO)$.
\end{defi}

\begin{prop}\label{Prop:O(P)-subalgebra}
    Let $P\subset X^+$ be a downward closed subset. Then ${}_R\bfO(P)$ is a subalgebra of ${}_R\bfO$ if and only if $P$ is also a submonoid of $X^+$.
\end{prop}
\begin{proof}
    The unit of ${}_R\bfO$ is $1_0^*$ and for any $\la,\mu\in P\subset X^+$, we have $1_\la^*\cdot1_\mu^*=1_{\la+\mu}^*$ in ${}_R\bfO$ by Lemma \ref{lem:filtration-O-multiplicative}. Thus if ${}_R\bfO(P)$ is a subalgebra then $\la+\mu\in P$ and $0\in P$ so that $P$ is a submonoid of $X^+$.
\end{proof}
\begin{lem}\label{lem:from Xplus to NN}[See \cite{Popov-contraction} p.324]
    There exists a homomorphism of additive monoids $h\colon X^+\to\NN$ such that $h(\la)<h(\mu)$ whenever $\la<\mu$.
\end{lem}
Fix a homomorphism $h\colon X^+\to\NN$ as above. Let $P\subset X^+$ be a downward closed submonoid. For any $n\in\NN$, denote $P_{\le n}\coloneqq\{\la\in P,h(\la)\le n\}$ and ${}_R\bfO_{\le n}(P)\coloneqq{}_R\bfO(P_{\le n})$. For $n\in\ZZ_{<0}$ we let ${}_R\bfO_{\le n}(P)=0$. Then ${}_R\bfO(P)$ becomes a $\ZZ$-filtered $R$-algebra and we can form the associated graded ring
\[\mathrm{gr}_h({}_R\bfO(P))\coloneqq\bigoplus_{n\in\NN}{}_R\bfO_{\le n}(P)/{}_R\bfO_{\le n-1}(P).\]
The following lemma follows from the definitions.
\begin{lem}\label{lem:grh}
    For any downward closed submonoid $P\subset X^+$, we have a canonical isomorphism of $R$-algebras $\mathrm{gr}({}_R\bfO(P))\cong\mathrm{gr}_h({}_R\bfO(P))$.
\end{lem}
For each $b\in\dot{\bf{B}}[\lambda]$, we denote the image of $b^*\in{}_R\bf{O}_{\le\lambda}$ in $\mathrm{gr}({}_R\bf{O})$ by $\bar{b^*}$. Then the elements $\{\bar{b^*},b\in\dot{\bf{B}}\}$ form a basis of $\mathrm{gr}({}_R\bf{O})$. For each $\la,\mu\in X^+$, $a\in\dot{\bf{B}}[\lambda]$ and $b\in\dot{\bf{B}}[\mu]$, the product of $\bar{a^*}$ and $\bar{b^*}$ in $\mathrm{gr}({}_R\bf{O})$ is given by
\[\bar{a^*}\bar{b^*}=\sum_{c\in\dot{\bf{B}}[\lambda+\mu]}\hat{m}_c^{ab}\bar{c^*}.\]

\subsection{Finiteness theorems}
\begin{defi}\label{def:G/U}
    Let ${}_R\mathbf{O}_{G/U}$ be the $R$-submodule of ${}_R\bf{O}$ spanned by the elements
    \[\bigcup_{\lambda\in X^+}\{b^*, b\in\dot{\bf{B}}_{0,\lambda}\}=\bigcup_{\la\in X^+}\{(b^-1_\la)^*, b\in\bfB(\la)\}.\]
    Let ${}_R\mathbf{O}_{U^-\backslash G}$ be the $R$-submodule of ${}_R\bf{O}$ spanned by the elements
    \[\bigcup_{\lambda\in X^+}\{b^*, b\in\sigma(\dot{\bf{B}}_{\lambda,0})\}=\bigcup_{\la\in X^+}\{(1_\la\sigma(b^+))^*,b\in\bfB(\la)\}.\]
    For any subset $P\subset X^+$, we denote
    \[{}_R\bfO_{G/U}(P)\coloneqq{}_R\bfO_{G/U}\bigcap{}_R\bfO(P),\quad{}_R\mathbf{O}_{U^-\backslash G}(P)\coloneqq{}_R\mathbf{O}_{U^-\backslash G}\bigcap{}_R\bfO(P).\]
\end{defi}
The notations used here will be justified in Proposition \ref{prop:U-invariant}. 
\begin{prop}\label{prop:basic-affine-space}
    Let $P$ be a submonoid of $X^+$. 
    \begin{enumerate}
        \item[(i)] The $R$-modules ${}_R\bfO_{G/U}(P)$ and ${}_R\bfO_{U^-\backslash G}(P)$ are $P$-graded subalgebras of ${}_R\bfO$.
        \item[(ii)] There is a natural embedding of $R$-algebras ${}_R\bfO_{G/U}(P)\into R[P]\otimes_R{}_R\bfO_{G/U}$ whose image is the degree $0$ subalgebra with respect to the $X$-grading that assigns $\la-\mu$ to the element $e^\la\otimes(b^-1_\mu)^*$, for any $\la\in P,\mu\in X^+$ and $b\in\bfB(\mu)$. Similarly, ${}_R\bfO_{U^-\backslash G}(P)$ is isomorphic to the degree $0$ subalgebra of $R[P]\otimes_R{}_R\bfO_{U^-\backslash G}$.
    \end{enumerate}
\end{prop}
\begin{proof}
    We only need to prove the statements for ${}_R\bfO_{G/U}(P)$ since then the case of ${}_R\bfO_{U^-\backslash G}(P)$ would follow by applying the involution $\omega\sigma$.\par 
    (i) Recall that $\dot{\bfB}_{0,0}=\dot{\bfB}[0]=\{1_0\}$. Since $0\in P$, the subset ${}_R\bfO_{G/U}(P)$ contains the multiplicative unit $1_0^*$. For any 
    $a\in\dot{\bfB}_{0,\la}, b\in\dot{\bfB}_{0,\la'}$ with $\la,\la'\in P$,  by \cite[1.14(a)]{LusztigGroupScheme} we see that the product $a^*b^*$ (in ${}_R\bfO$) is an $R$-linear combination of elements $c^*$ with $c\in\dot{\bfB}_{0,\la+\la'}$. This implies that
    ${}_R\bfO_{G/U}(P)$ is a subalgebra of ${}_R\bfO$ and moreover it is a $P$-graded algebra, in which the $\la$-component is the $R$-linear span of $\{b^*,b\in\dot{\bfB}_{0,\la}\}$. \par 
    (ii) We define the sought-for embedding by sending $(b^-1_\la)^*$ with $\la\in P$ and $b\in\bfB(\la)$ to the element $e^\la\otimes(b^-1_\la)^*$. This is an $R$-algebra homomorphism by Lemma \ref{lem:filtration-O-multiplicative} and  \cite[1.14(a)]{LusztigGroupScheme}, and it clearly satisfies all the requirements.
\end{proof}

\begin{lem}\label{lem:ideal-in-grO}
    Let $P\subset X^+$ be a submonoid. Let $I_P\subset\mathrm{gr}({}_R\bfO(P))$ be the $R$-submodule spanned by $\bigcup\limits_{\lambda\in P}\{\bar{b^*},b\in\dot{\bf{B}}[\lambda]\setminus\dot{\bf{B}}_{0,\lambda}\}$. Then $I_P$ is a two-sided ideal of $\rm{gr}({}_R\bfO(P))$ and we have a canonical ring isomorphism $\rm{gr}({}_R\bfO(P))/I_P\cong{}_R\bfO_{G/U}(P)$ sending $\bar{b^*}$ to ${b^*}$, for any $b\in\bigcup\limits_{\la\in P}\dot{\bfB}_{0,\la}$. \par 
    Similarly, let $I_P^-\subset\mathrm{gr}({}_R\bfO(P))$ be the $R$-submodule spanned by $\bigcup\limits_{\lambda\in P}\{\bar{b^*},b\in\dot{\bf{B}}[\lambda]\setminus\sigma(\dot{\bf{B}}_{\lambda,0})\}$. Then $I_P^-$ is a two-sided ideal of $\rm{gr}({}_R\bfO(P))$ and we have a canonical ring isomorphism $\rm{gr}({}_R\bfO(P))/I_P^-\cong{}_R\bfO_{U^-\backslash G}(P)$ sending $\bar{b^*}$ to ${b^*}$, for any $b\in\bigcup\limits_{\la\in P}\sigma(\dot{\bfB}_{\la,0})$. 
\end{lem}
\begin{proof}
    We only prove the first statement. The other will be similar. To show that $I_P$ is a two-sided ideal of $\rm{gr}({}_R\bfO(P))$, we may assume that $R=\A$ and it suffices to prove the following statement: for any $\lambda',\lambda''\in P$ with $\lambda=\lambda'+\lambda''$, any $a\in\dot{\bf{B}}[\lambda']$, $b\in\dot{\bf{B}}[\lambda'']$ and $c\in\dot{\bf{B}}_{0,\lambda}$, if either $a\notin\dot{\bfB}_{0,\la'}$ or $b\notin\dot{\bf{B}}_{0,\lambda''}$, then  $\hat{m}_c^{ab}=0$. Indeed, this translates to the statement that if either $\bar{a^*}\in I_P$ or $\bar{b^*}\in I_P$, then $\bar{a^*}\bar{b^*}\in I_P$, which then implies that $I_P$ is a two-sided ideal of $\rm{gr}({}_R\bfO(P))$. In the following proof for ease of typesetting, we denote $\hat{m}(c,a,b)\coloneqq\hat{m}_c^{ab}$ and all tensor products are over $R$.\par 
    We can write $c=c_1^-1_\lambda$ where $c_1\in\bf{B}(\lambda)$. Following \cite[\S4.26]{Lusztig-Infinity} we write
    \[\Bar{r}(c_1)=\sum_{c_{11},c_{12}\in\bf{B}}p(c_1,c_{11},c_{12})c_{11}\otimes  c_{12}\]
    where $p(c_1,c_{11},c_{12})\in\ZZ[v,v^{-1}]$ and vanishes for all but finitely many pairs $c_{11},c_{12}$. Then by \emph{loc.cit.} Lemma 3.10 we get            
    \[\Delta(c)=\Delta(c_1^-1_\lambda)=\sum_{\substack{c_{11},c_{12}\in\bf{B}\\ p(c_1,c_{11},c_{12})\ne0}}\sum_{\substack{\lambda',\lambda''\in X \\ \lambda'+\lambda''=\lambda}} \hat{m}(c,c_{11}^-1_{\lambda'},c_{12}^-1_{\lambda''})\cdot(c_{11}^-1_{\lambda'})\otimes (c_{12}^-1_{\lambda''})\]
    where the structure constant $\hat{m}(c,c_{11}^-1_{\lambda'},c_{12}^-1_{\lambda''})$ equals $p(c_1,c_{11},c_{12})$ times some explicit power of $v$.\par 
    We are interested in terms in the sum where $a=c_{11}^-1_{\lambda'}\in\dot{\bf{B}}[\lambda_1]$, $b=c_{12}^-1_{\lambda''}\in\dot{\bf{B}}[\lambda_2]$ with $\lambda_1,\lambda_2\in P$ and $\lambda_1+\lambda_2=\lambda$. Then $c_{11}^-1_{\la'}$ (resp. $c_{12}^-1_{\la''}$) acts nontrivially on $\La_{\la_1}$ (resp. $\La_{\la_2}$), hence $\lambda'\le\lambda_1$ and $\lambda''\le\lambda_2$. On the other hand, since $\lambda'+\lambda''=\lambda_1+\lambda_2$, we must have $\lambda'=\lambda_1\in X^+$ and $\lambda''=\lambda_2\in X^+$. By Lemma \ref{lem:B(lambda)-characterization} we get that $c_{11}\in\bf{B}(\lambda')$ and $c_{12}\in\bf{B}(\lambda'')$, and hence $a=c_{11}^-1_{\lambda'}\in\dot{\bf{B}}_{0,\lambda'}$ and $b=c_{12}^-1_{\lambda''}\in\dot{\bf{B}}_{0,\lambda''}$. This finishes the proof.
\end{proof}

Let $P\subset X^+$ be a submonoid. Define an $R$-linear map 
\begin{equation}\label{eq:phi}
    \varphi_P\colon\rm{gr}({}_R\bfO(P))\to{}_R\bfO_{G/U}(P)\otimes_R{}_R\bfO_{U^-\backslash G}(P)
\end{equation}
by the formula
\[\varphi_P(\bar{c^*})=\sum_{a\in\dot{\bf{B}}_{0,\lambda}, b\in\sigma(\dot{\bf{B}}_{\lambda,0})}m_{ab}^{c} a^*\otimes  b^*,\quad\forall c\in\dot{\bf{B}}[\lambda],\lambda\in P.\]

The tensor product ${}_R\bfO_{G/U}(P)\otimes_R{}_R\bfO_{U^-\backslash G}(P)$ has the structure of an $X$-graded algebra where the grading is defined as follows: for $\la,\mu\in P$ and $b_1\in\bfB(\la)$, $b_2\in\bfB(\mu)$, the element $(b_1^-1_\la)^*\otimes(1_\mu\sigma(b_2)^+)^*$ has degree $\la-\mu\in X$. 

\begin{prop}\label{prop:grO-decription}
    The map $\varphi_P$ in \eqref{eq:phi} induces an $R$-algebra isomorphism from $\mathrm{gr}({}_R\bfO(P))$ onto the degree $0$ subalgebra of 
    \[{}_R\bfO_{G/U}(P)\otimes_R{}_R\bfO_{U^-\backslash G}(P),\]
    with respect to the $X$-grading defined above. 
\end{prop}
\begin{proof}
    We may assume that $R=\A$. First we show that $\varphi_P$ is a ring homomorphism. Let $c\in\dot{\bf{B}}[\lambda]$ and $d\in\dot{\bf{B}}[\lambda']$ with $\lambda,\lambda'\in P$. One calculates that 
    \begin{equation}
        \begin{split}
            \varphi_P(\bar{c^*})\varphi_P(\bar{d^*})&=(\sum_{a\in\dot{\bf{B}}_{0,\lambda}, b\in\sigma(\dot{\bf{B}}_{\lambda,0})}m_{ab}^{c} a^*\otimes  b^*)(\sum_{e\in\dot{\bf{B}}_{0,\lambda'}, f\in\sigma(\dot{\bf{B}}_{\lambda',0})}m_{ef}^{d} e^*\otimes  f^*)\\
            &=\sum_{\substack{a\in\dot{\bf{B}}_{0,\lambda}, b\in\sigma(\dot{\bf{B}}_{\lambda,0})\\ e\in\dot{\bf{B}}_{0,\lambda'}, f\in\sigma(\dot{\bf{B}}_{\lambda',0})}}m_{ab}^c m_{ef}^d a^*e^*\otimes  b^*f^*\\
            &=\sum_{\substack{a\in\dot{\bf{B}}_{0,\lambda}, b\in\sigma(\dot{\bf{B}}_{\lambda,0})\\ e\in\dot{\bf{B}}_{0,\lambda'}, f\in\sigma(\dot{\bf{B}}_{\lambda',0})}}\sum_{z\in\dot{\bf{B}}_{0,\lambda+\lambda'},w\in\sigma(\dot{\bf{B}}_{\lambda+\lambda',0})}m_{ab}^cm_{ef}^d\hat{m}_z^{ae}\hat{m}_w^{bf}z^*\otimes  w^*,
        \end{split}
    \end{equation}
    and
    \begin{equation}
        \begin{split}
            \varphi_P(\bar{c^*}\bar{d^*})&=\varphi_P(\sum_{x\in\dot{\bf{B}}[\lambda+\lambda']}\hat{m}_x^{cd}\bar{x^*})\\&=\sum_{x\in\dot{\bf{B}}[\lambda+\lambda']}\sum_{z\in\dot{\bf{B}}_{0,\lambda+\lambda'},w\in\sigma(\dot{\bf{B}}_{\lambda+\lambda',0})}\hat{m}_x^{cd}m_{zw}^x z^*\otimes  w^*.
        \end{split}
    \end{equation}
    For fixed $c,d,z,w\in\dot{\bf{B}}$, the structure constants satisfy the equation
    \begin{equation}\label{eq:structure constants}
\sum_{a,b,e,f\in\dot{\bf{B}}}m_{ab}^cm_{ef}^d\hat{m}^{ae}_z\hat{m}^{bf}_w=\sum_{x\in\dot{\bf{B}}}\hat{m}_x^{cd}m_{zw}^x
    \end{equation}
    which follows from the fact that $\hat{\Delta}\colon\hat{\bf{U}}\to\hat{\bfU}\hat{\otimes}\hat{\bfU}$ is a ring homomorphism. 
    In our situation, the elements satisfy further constraints:
    \begin{enumerate}
        \item[(i)] $c\in\dot{\bf{B}}[\lambda]$ and $d\in\dot{\bf{B}}[\lambda']$;
        \item[(ii)] $z\in\dot{\bf{B}}_{0,\lambda+\lambda'}\subset\dot{\bf{B}}[\lambda+\lambda']$ and $w\in\sigma(\dot{\bf{B}}_{\lambda+\lambda',0})\subset\dot{\bf{B}}[\lambda+\lambda']$.
    \end{enumerate}
    Now from (i) and Lemma \ref{lem:filtration-O-multiplicative} we get:
    \begin{enumerate}
        \item[(iii)]  $\hat{m}_x^{cd}\ne0$ only if $x\in\dot{\bf{B}}[\le(\lambda+\lambda')]$;
    \end{enumerate}
    From (ii) we get:
    \begin{enumerate}
        \item[(iv)] $m_{zw}^x\ne0$ only if $x\in\dot{\bf{B}}[\ge(\lambda+\lambda')]$.
    \end{enumerate}
    On the other hand, (i) implies that: 
    \begin{enumerate}
        \item[(v)] $m_{ab}^cm_{ef}^d\ne0$ only if $a,b\in\dot{\bf{B}}[\le\lambda]$ and $e,f\in\dot{\bf{B}}[\le\lambda']$.
    \end{enumerate}
    Combining (v) and (ii), we get that:
    \begin{enumerate}
        \item[(vi)]  $m_{ab}^cm_{ef}^d\hat{m}^{ae}_z\hat{m}^{bf}_w\ne0$ only if $a,b\in\dot{\bf{B}}[\lambda]$ and $e,f\in\dot{\bf{B}}[\lambda']$.
    \end{enumerate}
    Combining (vi) and Lemma \ref{lem:ideal-in-grO} we get that: 
    \begin{enumerate}
        \item[(vii)] $m_{ab}^cm_{ef}^d\hat{m}^{ae}_z\hat{m}^{bf}_w\ne0$ only when the following are all satisfied:
        \[a\in\dot{\bf{B}}_{0,\lambda}, b\in\sigma(\dot{\bf{B}}_{\lambda,0}), e\in\dot{\bf{B}}_{0,\lambda'}, f\in\sigma(\dot{\bf{B}}_{\lambda',0}).\]
    \end{enumerate}
    Finally, combining (iii), (iv), (vii) and \eqref{eq:structure constants} we conclude that \[\varphi_P(\bar{c^*})\varphi_P(\bar{d^*})=\varphi_P(\bar{c^*}\bar{d^*})\] 
    and this shows that $\varphi_P$ is a ring homomorphism.\par 
    Now let $\la\in P$ and $c_1,c_2\in\bfB(\la)$. Let $c\coloneqq\beta_\la(c_1,c_2)\in\dot{\bfB}[\la]$, where the map $\beta_\la$ is defined in Proposition \ref{prop:parametrize-B[la]}. Then we have $c-c_1^-1_\la\sigma(c_2)^+\in\dot{\bfU}[>\la]$ and hence
    \[\varphi_P(\bar{c^*})=(c_1^-1_\lambda)^*\otimes (1_\lambda\sigma(c_2)^+)^*.\]
    By Proposition \ref{prop:parametrize-B[la]} we see that $\varphi_P$ is injective and its image is the $\A$-submodule spanned by the set
    \[\bigcup_{\lambda\in P}\{a^*\otimes  b^*, (a,b)\in\dot{\bf{B}}_{0,\lambda}\times\sigma(\dot{\bf{B}}_{\lambda,0})\}.\]
    This is precisely the degree $0$ subalgebra of ${}_\A\bfO_{G/U}(P)\otimes_\A{}_\A\bfO_{U^-\backslash G}(P)$.
\end{proof}

\begin{cor}\label{cor:grO-description}
    Let $P\subset X^+$ be a submonoid. Then the $R$-algebra 
    \[\mathrm{gr}({}_R\bfO(P))\] 
    is isomorphic to the degree $(0,0)$ subalgebra of the $X\oplus X$-graded $R$-algebra
    \[R[P]\otimes{}_R\bfO_{G/U}\otimes{}_R\bfO_{U^-\backslash G}\]
    in which the grading is defined as follows: for any $\la\in P$, $\mu,\nu\in X^+$ and $b\in\bfB(\mu)$, $b'\in\bfB(\nu)$, the element $e^\la\otimes (b^-1_\mu)*\otimes (1_\nu\sigma(b')^+)^*$ has degree $(\la-\mu,\la-\nu)$. Here $R[P]$ is the monoid algebra of $P$ with $R$-basis $\{e^\la,\la\in P\}$.
\end{cor}
\begin{proof}
    This is a combination of Proposition \ref{prop:grO-decription} and Proposition \ref{prop:basic-affine-space}. We have a canonical $R$-algebra embedding
    \[\iota_P\colon\mathrm{Im}(\varphi_P)=({}_R\bfO_{G/U}(P)\otimes_R{}_R\bfO_{U^-\backslash G}(P))_0\to R[P]\otimes{}_R\bfO_{G/U}\otimes{}_R\bfO_{U^-\backslash G}\]
    defined on homogeneous elements by 
    \[\iota_P((b_1^-1_\la)^*\otimes(1_\la\sigma(b_2)^+)^*)=e^\la\otimes(b_1^-1_\la)^*\otimes(1_\la\sigma(b_2)^+)^*,\quad b_1,b_2\in\bfB(\la), \la\in P.\]
    The image of $\iota_P$ is the degree $(0,0)$ subalgebra and the composition $\iota_P\circ\varphi_P$ gives the desired algebra isomorphism.  
\end{proof}

\begin{cor}\label{cor:grO-finite}
    Let $P$ be a finitely generated submonoid of $X^+$ that is downward closed. Then the $R$-algebras ${}_R\bfO_{G/U}(P)$, ${}_R\bfO_{U^-\backslash G}(P)$, $\mathrm{gr}({}_R\bfO(P))$ and ${}_R\bfO(P)$ are all finitely generated. 
\end{cor}
\begin{proof}
    Let $\lambda_1,\dotsc,\lambda_m$ be a set of generators of the monoid $P$. Consider the sets 
    \[\Gamma_P^+=\bigcup_{j=1}^m\dot{\bf{B}}_{0,\lambda_j},\qquad\Gamma_P^-=\bigcup_{j=1}^m\sigma(\dot{\bf{B}}_{\lambda_j,0})\]
    which are both subsets of the finite set $\Gamma\coloneqq\bigcup_{j=1}^m\dot{\bfB}[\la_j]$. \par   
    For any $\lambda',\lambda''\in P$ and any $c\in\dot{\bf{B}}_{0,\lambda}$ where $\lambda=\lambda'+\lambda''$, by Lemma \ref{lem:h_c-identity} there exists a function $h_c\colon\dot{\bf{B}}_{0,\lambda'}\times\dot{\bf{B}}_{0,\lambda''}\to\ZZ[v^{\pm1}]$ such that the following equality holds in ${}_\A\bf{O}$:
    \[\sum_{a\in\dot{\bf{B}}_{0,\lambda'},b\in\dot{\bf{B}}_{0,\lambda''}}h_c(a,b)a^*b^*=c^*.\]
    From this we deduce by induction that $\{b^*,b\in\Gamma_P^+\}$ is a generating set for the $R$-algebra ${}_R\bfO_{G/U}(P)$. Similarly (or by applying the involution $\omega\sigma$), we see that $\{b^*,b\in\Gamma_P^-\}$ is a generating set for the $R$-algebra ${}_R\bfO_{U^-\backslash G}(P)$. Then the degree $0$ subalgebra of ${}_R\bfO_{G/U}(P)\otimes_R{}_R\bfO_{U^-\backslash G}(P)$ is generated by the finite set 
    \[\bigcup_{j=1}^m\dot{\bfB}_{0,\la_j}^*\otimes\sigma(\dot{\bfB}_{\la_j,0})^*.\]
    Then by Proposition \ref{prop:grO-decription} we see that the finite set $\{\bar{b^*},b\in\Gamma\}$ generates the algebra $\rm{gr}({}_R\bfO(P))$. Lifting generators, and using induction thanks to Lemma \ref{lem:from Xplus to NN} and Lemma \ref{lem:grh} we deduce that the finite set $\{b^*,b\in\Gamma\}$ generates the $R$-algebra ${}_R\bfO(P)$. 
\end{proof}

\subsection{Normality}
\begin{lem}\label{lem:normal}
    Let $A$ be a normal integral domain and $X$ a scheme over $\Spec(A)$. Suppose that the structure morphism $X\to\Spec(A)$ is flat and of finite presentation, and all its geometric fibers are integral and normal. Then the $X$ is integral and normal and its ring of global functions $\Gamma(X,\O_X)$ is a normal integral domain. 
\end{lem}
\begin{proof}
    By \cite[11.3.13]{EGAIV8a15} the scheme $X$ is integral and normal. Then we conclude by \cite[Lemme 8.8.6.1]{EGAII1a8} or \cite[\href{https://stacks.math.columbia.edu/tag/0358}{Lemma 0358}]{Stacks}.
\end{proof}
\begin{lem}\label{lem:filtered-ring-normal}
    Let $A=\bigcup_{n\in\NN}A_n$ be a filtered commutative unital ring such that 
    \begin{itemize}
        \item For any $n\in\NN$, $A_n$ is a subgroup of the additive group of $A$;
        \item $A_n\subset A_{n+1}$ for all $n\in\NN$;
        \item $A_n\cdot A_m\subset A_{n+m}$ for all $m,n\in\NN$.
    \end{itemize}
    Let $A_{-1}=\{0\}$. Suppose that the associated graded ring $\mathrm{gr}(A)=\bigoplus_{n\in\NN}A_n/A_{n-1}$ is a noetherian normal integral domain. Then $A$ is a normal integral domain. 
\end{lem}
\begin{proof}
    For any $a\in A$, let $h(a)\coloneqq\min\{n\in\NN,a\in A_n\}$ and let $\bar{a}\in\mathrm{gr}(A)$ be the image of $a$ in $A_{h(a)}/A_{h(a)-1}$. Then $\bar{a}=0$ if and only if $a=0$. From the assumption that $\mathrm{gr}(A)$ is an integral domain one easily verifies that for any $a,b\in A$ we have $h(ab)=h(a)+h(b)$ and $\overline{ab}=\bar{a}\bar{b}$. Then we deduce that $A$ is an integral domain.\par
    Let $0\ne \alpha=\frac{x}{y}\in\mathrm{Frac}(A)$ be integral over $A$. Then there exists $0\ne c\in A$ such that $a_n\coloneqq c\alpha^n\in A$ for all $n\ge0$ by \cite[\href{https://stacks.math.columbia.edu/tag/00GX}{Lemma 00GX}]{Stacks}. Then we have $\bar{c}(\bar{x}/\bar{y})^n\in\mathrm{gr}(A)$ for all $n\ge0$. Since $\mathrm{gr}(A)$ is noetherian and integrally closed, we get that $\bar{x}/\bar{y}\in\mathrm{gr}(A)$, again by \cite[\href{https://stacks.math.columbia.edu/tag/00GX}{Lemma 00GX}]{Stacks}. Therefore we can find a homogeneous lift $a\in A$ of $\bar{x}/\bar{y}$, of degree $h(a)=h(x)-h(y)$. So we have $h(x-ya)<h(x)$. Replacing $\alpha$ by $\alpha-a$ and (since $\alpha-a$ remains integral) repeating the above argument, we finally conclude that $\alpha\in A$ and hence $A$ is integrally closed. 
\end{proof}
\begin{theo}\label{theo:O(L)-normal}
    Let $\k$ be a noetherian normal integral domain, viewed as $\A$-algebra by $v\mapsto1$. Let $\L$ be a saturated finitely generated submonoid of $X^+$. Then the $\k$-algebras ${}_k\bfO_{G/U}(\L)$, ${}_k\bfO_{U^-\backslash G}(\L)$ and $\mathrm{gr}({}_\k\bfO(\L))$ are noetherian normal integral domains. If moreover $\L$ is downward closed, then the same holds for ${}_\k\bfO(\L)$.
\end{theo}
\begin{proof}
    Since the schemes $G/U$ and $U^-\backslash G$ are both smooth over $\k$ and $\k$ is normal, the algebras $\k[G/U]$ and $\k[U^-\backslash G]$ are normal integral domains by Lemma \ref{lem:normal}.\par 
    By Lemma \cite[Proposition 3.4.1]{Ogus-Log}, the monoid algebra $\k[\L]$ is integrally closed. By Proposition \ref{prop:basic-affine-space}, the $\k$-algebra ${}_\k\bfO_{G/U}(\L)$ is the degree $0$ subalgebra of the $X$-graded integrally closed $\k$-algebra $\k[\L]\otimes_\k\k[G/U]$, hence is integrally closed. Similarly ${}_\k\bfO_{U^-\backslash G}(\L)$ is integrally closed. \par 
    By Corollary \ref{cor:grO-finite}, the $\k$-algebra $\mathrm{gr}({}_\k\bfO(\L))$ is finitely generated and by Corollary \ref{cor:grO-description} it is isomorphic to the degree $(0,0)$ subalgebra of the $X\oplus X$-graded $\k$-algebra $\k[\L]\otimes_\k\k[G/U]\otimes_\k\k[U^-\backslash G]$. The latter algebra is integrally closed since it is the ring of global functions on a scheme that admits a smooth morphism to the integral normal scheme $\Spec(\k[\L])$. Hence its degree $(0,0)$ subalgebra $\mathrm{gr}({}_\k\bfO(\L))$ is also integrally closed. Finally when $\L$ is also downward closed, by Lemma \ref{lem:grh} and Lemma \ref{lem:filtered-ring-normal} we deduce that ${}_\k\bfO(\L)$ is a normal integral domain. By Corollary \ref{cor:grO-finite} it is finitely generated over $\k$ and hence noetherian. 
\end{proof}

\section{Representations and invariants}\label{sec:reps-and-invariants}
In this section, we fix a commutative base ring $\k$ which is viewed as an $\A$-algebra through $v\mapsto1$.
\subsection{Review on algebraic representations}
Let $H$ be an affine algebraic group over $\k$ with coordinate ring $\k[H]$, which is a finitely generated Hopf $\k$-algebra. The \emph{distribution algebra}\footnote{For us, a distribution is an element of the \emph{full} linear dual of $\k[H]$. This differs from the convention in \cite{Jantzen2003}.} of $H$ is the $\k$-algebra ${}_\k\hat{\bfU}(H)\coloneqq\Hom_\k(\k[H],\k)$, in which the multiplication is induced by the coproduct on $\k[H]$, and equipped with a counit map $\varepsilon\colon{}_\k\hat{\bfU}(H)\to\k$ defined by evaluation at $1\in\k[H]$. Let ${}_\k\hat{\bfU}(H)_+\coloneqq\ker(\varepsilon)$ be the \emph{augmentation ideal}. When $H$ is a split reductive group, this recovers the classical limit ($v\mapsto 1$) of Definition \ref{defi:hatU}. As explained there, we equip ${}_\k\hat{\bfU}(H)$ with the weak* topology.\par
Let $\Rep_H$ be the category of algebraic representations of $H$ on $\k$-modules. By definition, this is the category of $\k[H]$-comodules. In other words, an object in $\Rep_H$ is a $\k$-module $M$ equipped with a $\k$-linear \emph{co-action} map 
\[\mu_M\colon M\to \k[H]\otimes_\k M\] 
satisfying the following two conditions:
\begin{itemize}
    \item $(\mathrm{Id}_{\k[H]}\otimes\mu_M)\circ\mu_M=(\Delta_{\k[H]}^{\mathrm{op}}\otimes\mathrm{Id}_M)\circ\mu_M$,
    \item $(\epsilon_{\k[H]}\otimes\mathrm{Id}_M)\circ\mu_M=\mathrm{Id}_M$.
\end{itemize}
Here $\Delta_{\k[H]}$ (resp. $\epsilon_{\k[H]}$) denotes the coproduct (resp. counit) of the Hopf algebra $\k[H]$.
These conditions ensure that $\mu_M$ induces a ring homomorphism ${}_\k\hat{\bfU}(H)\to\mathrm{End}_{\k}(M)$ sending an element $\xi\in{}_\k\hat{\bfU}(H)=\mathrm{Hom}_{\k}(\k[H],\k)$ to the $\k$-linear endomorphism $(\xi\otimes\mathrm{\mathrm{Id}_M})\circ\mu_M$ of $M$. For any $m\in M$, there exists a finitely generated $\k$-submodule $V_m\subset\k[H]$ such that $\mu_M(m)\in V_m\otimes_\k M$ (see for example \cite[VI$_\text{B}$, Lemme 11.8]{SGA3-1}). Thus the annihilator of $m$ in ${}_\k\hat{\bfU}(H)$ contains the open (under the weak* topology) subset $\{\xi\in{}_\k\hat{\bfU}(H),\xi(V_m)=0\}$. This shows that $M$ is a continuous ${}_\k\hat{\bfU}(H)$-module and so we obtain a natural functor
\[\Rep_H\to{}_\k\hat{\bfU}(H)\text{-}\mod.\]
Now suppose that $\k[H]$ is a free $\k$-module, and choose a basis $\Phi$. For any $\varphi\in\Phi$, let $\varphi^*\in{}_\k\hat{\bfU}(H)$ be the element whose value equals to $1$ at $\varphi$ and equals to $0$ at any other element in $\Phi$. Then for any continuous ${}_\k\hat{\bfU}$-module $M$, we can define a $\k$-linear map 
\[\mu_M\colon M\to \k[H]\otimes_\k M,\quad\mu_M(m)\coloneqq\sum_{\varphi\in\Phi}\varphi\otimes \varphi^*\cdot m\]
where the sum is finite by continuity. This gives $M$ the structure of a $\k[H]$-comodule. In this way we get a functor 
\[{}_\k\hat{\bfU}(H)\text{-}\mod\to\Rep_H.\]
The two functors defined above both commute with the forgetful functor to the category of $\k$-modules, and are easily seen to be inverse to each other. In summary we get:
\begin{prop}\label{prop:G-rep-equiv}
    Let $H$ be an affine algebraic group over $\k$ such that $\k[H]$ is a free $\k$-module. Then the functors defined above induce an equivalence of categories
    \[\Rep_H\simeq{}_\k\hat{\bfU}(H)\text{-}\mod.\]
\end{prop}
Now we fix a based root datum $\Psi$ and the associated split reductive group scheme $(G,B,T)$ over $\k$ as in \S\ref{sec:notation-reductive-group}. Combining Proposition \ref{prop:G-rep-equiv} and Proposition \ref{prop:uni-mod-equiv} we obtain equivalences of categories
\[\Rep_G^{ft}\simeq{}_\k\hat{\bfU}\text{-}\mod^{ft}\simeq{}_\k\dot{\bfU}\text{-}\mod^{ft}\]
in which all functors commute with the natural forgetful functor to $\k\text{-}\mod^X$, the category of $X$-graded $\k$-modules (which is nothing but $\Rep_T$). Here $\Rep_G^{ft}\subset\Rep_G$ is the full subcategory of $G$-modules that are finitely generated as $\k$-modules. This provides the link between the representation theory of quantum enveloping algebras and reductive algebraic groups. Let us recall some basic facts from the latter theory. \par 
We have the induction functor $\Ind_B^G\colon \Rep_B\to\Rep_G$ defined by 
\[\Ind_B^GM=(\k[G]\otimes_\k M)^B,\quad\forall M\in\Rep_B\]
where $B$ (resp. $G$) acts on $\k[G]$ by right (resp. left) regular representation.
\begin{lem}\label{lem:G-invariant}
    For each $G$-module $N$ that is flat over $\k$, the natural map $\k[G]\otimes_\k N\to N$ defined by $f\otimes n\mapsto f(1)n$ induces an isomorphism of $G$-modules $\Ind_B^GN\cong N$, where on the left hand side $N$ is viewed as a $B$-module (via the natural forgetful functor $\mathrm{Res}^G_B\colon\mathrm{Rep}_G\to\Rep_B$). Moreover, we have $N^B=N^G$. 
\end{lem}
\begin{proof}
    Let $N_1=N$ (resp. $N_2=N$) be the $B\times G$-module on which $G$ (resp. $B$) acts trivially and $B$ (resp. $G$) acts as usual. Consider the composition
    \[\varphi_N\colon\k[G]\otimes_\k N_1\xrightarrow{\mathrm{Id}_{\k[G]}\otimes\mu_N}\k[G]\otimes_\k\k[G]\otimes_\k N_1\xrightarrow{m\otimes\mathrm{Id}_N}\k[G]\otimes_\k N_2\]
    where $\mu_N$ is the co-action map and $m$ is the multiplication map. Informally, if we view elements in $\k[G]\otimes_\k N$ as maps $f\colon G(\k)\to N$, then $\varphi_N(f)(x)=x\cdot f(x)$ for all $x\in G(\k)$ and the formula is functorial on points over every $\k$-algebra. If we view $\k[G]$ as a $B\times G$-module where $B$ (resp. $G$) acts by right (resp. left) regular representation, then $\varphi_N$ is an isomorphism of $B\times G$-modules. Taking $B$-invariants, we get an isomorphism of $G$-modules
    \[\Ind_B^GN=(\k[G]\otimes_\k N_1)^B\cong(\k[G]\otimes_\k N_2)^B=\k[G]^B\otimes_k N=N\]
    where the middle equality follows from the $\k$-flatness of $N$ and the last equality follows from the well-known fact that $\k[G]^B=\k$. Moreover, taking $G$-invariants we get
    \[N^B=(\Ind_B^GN)^G=N^G.\]
\end{proof}
For any $\lambda \in X^+$, we define 
$$S_{\lambda}\coloneqq\mathrm{Ind_{B^-}^G(\k_{\lambda}})\simeq \mathrm{Ind_{B}^G(\k_{w_0(\lambda)}})\simeq (\k_{w_0(\lambda)}\otimes  \k[G])^B,$$
$$W_{\lambda}\coloneqq(S_{-w_0(\lambda)})^{*}\simeq\mathrm{Ind_{B}^G(\k_{-\lambda}})^{*}$$ 
to be respectively the \emph{Schur module} with highest weight $\lambda$ and the \emph{Weyl module} with highest weight $\lambda$.\par
For any $G$-module $M$, we define 
\[M^{(\lambda)}\coloneqq(M^U)_{\lambda}\simeq \Hom_{B}(\k_{\lambda},M)\simeq (\k_{-\lambda}\otimes  M)^B.\]

Recall the universal property of Weyl modules (see \cite[Proposition 21]{Kallen2010} which generalizes \cite[II 2.13]{Jantzen2003} stated over fields): if $\lambda$ is dominant, we have natural isomorphisms of $\k$-modules 
\[M^{(\lambda)}\simeq \Hom_G(\k,M\otimes S_{-w_0(\lambda)})\simeq  \Hom_G(W_{\lambda},M).\]
\begin{prop}\label{prop:Lambda is Weyl}
    For any $\lambda \in X^+$, we have a natural isomorphism of $G$-modules
    $$W_{\lambda}\xrightarrow{\sim} {}_{\k}\Lambda_{\lambda}.$$ As a consequence, the Weyl modules are endowed with a continuous ${}_{\k}\hat{\bfU}$-module structure and a unital ${}_{\k}\dot{\bfU}$-module structure. 
\end{prop}
\begin{proof}
     Note that for any $\k$-algebra $A$ we have isomorphisms of $G_A$-modules (see \cite[Proposition 16]{Kallen2010}):
    \[S_{\lambda}\otimes_\k A\cong\mathrm{Ind_{B_A^-}^{G_A}(A_{\lambda}}),\quad W_{\lambda}\otimes_\k A\cong \Hom_A(S_{-w_0(\lambda)}\otimes_\k A,A).\]
    Thus we may assume that $\k=\ZZ$. According to the universal property of Weyl modules, we have a unique $G$-module homomorphism $\varphi\colon  W_{\lambda}\to {}_{\ZZ}\Lambda_{\lambda}$ sending a highest weight vector $v_{\lambda}$ of $W_{\lambda}$ to $\eta_{\lambda}$. It is injective since $W_{\lambda}\otimes_\ZZ\QQ$ is simple and $W_{\lambda}$ is free of finite type. By Proposition \ref{prop:G-rep-equiv}, $W_\la$ has the structure of ${}_\ZZ\hat{\bfU}$-module such that the map $\varphi$ is ${}_\ZZ\hat{\bfU}$-equivariant. Since $\eta_{\lambda}$ generates ${}_{\ZZ}\Lambda_{\lambda}$ as a ${}_\ZZ\hat{\bfU}$-module, we see that $\varphi$ is surjective and hence an isomorphism.
\end{proof}
Because of this result, we will also refer to ${}_\k\La_\la$ as the Weyl module of highest weight $\la$.  
\subsection{Invariants}
\begin{defi}
    Let $A$ be an $R$-algebra equipped with an $R$-algebra homomorphism $\varepsilon\colon A\to R$ (for instance $A$ could be a Hopf algebra, or more generally an augmented $R$-algebra). Let $M$ be an $A$-module. The submodule of $A$-invariants in $M$ is defined by
    \[H^0(A,M)\coloneqq\mathrm{Hom}_{A\text{-}\mathrm{mod}}(R,M)=\{m\in M\mid am=0,\;\forall a\in\ker(\varepsilon)\}\]
    where $R$ is viewed as an $A$-module via $\varepsilon$. 
\end{defi}
Let $H$ be an affine algebraic group over $\k$ with coordinate ring $\k[H]$ and distribution algebra ${}_\k\hat{\bfU}(H)=\mathrm{Hom}_{\k}(\k[H],\k)$. Let ${}_\k\hat{\bfU}(H)_+\coloneqq\{\xi\in{}_\k\hat{\bfU}(H),\xi(1)=0\}$ be the augmentation ideal. Let $M\in\Rep_H$ be a $\k[H]$-comodule with co-action map $\mu_M\colon M\to\k[H]\otimes_\k M$. The submodule of $H$-invariants in $M$ is defined by
\[M^H\coloneqq\ker(\mu_M-1\otimes\mathrm{Id}_M)=\{m\in M\mid\mu_M(m)=1\otimes m\}.\]
Recall that the co-action map $\mu_M$ induces a ${}_\k\hat{\bfU}(H)$-module structure on $M$.
\begin{lem}\label{lem:invariant-description}
    Suppose that $\k[H]$ is a free $\k$-module. Then for any $\k[H]$-comodule $M$ we have
    \[M^H=H^0({}_\k\hat{\bfU}(H),M)\coloneqq\Set{m\in M\mid \xi\cdot m=0,\forall\xi\in{}_\k\hat{\bfU}(H)_+}.\]
\end{lem}
\begin{proof}
    If an element $m\in M$ satisfies $\mu_M(m)=1\otimes m$, then for any element $\xi\in{}_\k\hat{\bfU}(H)_+$ we have $\xi\cdot m=\xi(1)m=0$. This shows that $M^H\subset H^0({}_\k\hat{\bfU}(H),M)$.\par 
    Conversely suppose $m\in H^0({}_\k\hat{\bfU}(H),M)$. Put $d=\mu_M(m)-1\otimes m$. For every $\k$-linear functional $\xi\colon \k[H]\to\k$, the functional $\xi-\xi(1)\varepsilon_{\k[H]}$ belongs to the augmentation ideal ${}_\k\hat{\bfU}(H)_+$. So we have $(\xi\otimes\mathrm{Id}_M)(d)=0$. Since $\k[H]$ is free, its basis coordinate functionals separate $\k[H]\otimes_\k M$ and therefore $d=0$. 
\end{proof}

\begin{ex}
Let $H=T$ be a split torus over $\k$ with character lattice $X=X^*(T)$. Then $\k[T]=\k[X]$ is the group algebra of $X$ and ${}_\k\hat{\bfU}(T)=\prod_{\la\in X}\k 1_\la$ where $\{1_\la,\la\in X\}$ are mutually orthogonal idempotents. The augmentation ideal is ${}_\k\hat{\bfU}(T)_+=\prod_{\la\in X\setminus0}\k 1_\la$.\par 
The category $\Rep_T$ is canonically equivalent to the category $\k\text{-}\mod^X$ of $X$-graded $\k$-modules. Indeed, for any $M\in\Rep_T$, we have a decomposition $M=\bigoplus\limits_{\la\in X}M_\la$ where $M_\la\coloneqq1_\la M$ and conversely a $\k$-module with such a decomposition labeled by $X$ is naturally a ${}_\k\hat{\bfU}(T)$-module on which $1_\la$ acts as the projector to the $\la$ component.\par 
For any $M\in\Rep_T$ we have
\[M^T=\{m\in M\mid 1_\la m=0,\;\forall\la\in X\setminus0\}=M_0.\]
More generally, for any $\zeta\in X$ we let $\k_\zeta=\k$ with the ${}_\k\hat{\bfU}(T)$-module structure in which $1_\la$ acts by $\delta_{\la\zeta}$, then we have
$(M\otimes_\k \k_{-\la})^T=M_\la$.
\end{ex}
Now we consider the left (resp. right) regular representations on $M=\k[H]$: 
\[\rho_l,\rho_r\colon{}_\k\hat{\bfU}(H)\to\mathrm{End}_{\k\text{-}\mod}(\k[H]).\]
The co-action map for the left regular representation is 
\[\mu_l\coloneqq(S_{\k[H]}\otimes\mathrm{Id})\circ\Delta_{\k[H]}\colon\k[H]\to\k[H]\otimes_\k\k[H]\]
where $S_{\k[H]}$ is the antipode of $\k[H]$. On the other hand, the co-action map for the right regular representation is 
\[\mu_r={}^t\Delta_{\k[H]}\coloneqq\mathrm{sw}\circ\Delta_{\k[H]}\colon\k[H]\to\k[H]\otimes_\k\k[H]\]
where $\mathrm{sw}$ is the $\k$-linear automorphism of $\k[H]\otimes_\k\k[H]$ switching the two factors.\par 
Then for any $\xi,\xi'\in{}_\k\hat{\bfU}(H)$ and $\varphi\in\k[H]$, we have
\[\langle\rho_l(\xi)\cdot\varphi,\xi'\rangle=\langle\varphi,S(\xi)\xi'\rangle,\quad\langle\rho_r(\xi)\cdot\varphi,\xi'\rangle=\langle\varphi,\xi'\xi\rangle.\]
Now let us review some classical results from invariant theory. First we have the following finiteness theorem:
\begin{theo}\label{theo:Kallen-finiteness}
    Let $k$ be an algebraically closed field, $G$ be a reductive group over $k$ and $U\subset G$ be the unipotent radical of a Borel subgroup of $G$. Let $A$ be a finitely generated commutative $k$-algebra on which $G$ acts rationally by $k$-algebra automorphisms. Then the invariant $k$-algebras $A^G$ and $A^U$ are finitely generated. 
\end{theo}
\begin{proof}
    See \cite[Theorem A.1.0]{GIT} for the finite generation of $A^G$ and \cite[Theorem 9]{Grosshans1992} for the finite generation of $A^U$. 
\end{proof}
It is worth mentioning that this result has been generalized to any noetherian base ring by Franjou and van der Kallen \cite[Theorem 3, Lemma 25]{Kallen2010}, although we will not use the more general version. Let us review some other results of van der Kallen that we will use.
\begin{lem}\label{lem:transfer-principle}
    Let $\k$ be a commutative ring, $G$ be a reductive group over $\k$ and $U\subset G$ the unipotent radical of a Borel subgroup of $G$. Let $M$ be a $G$-module. Then we have a natural isomorphism $M^U\simeq(M\otimes_\k \k[G/U])^G$ where $G$ acts diagonally on the right hand side.
\end{lem}
\begin{proof}
    See \cite[Proof of Proposition 5]{Kallen2014}.
\end{proof}
To state the next result, recall that the \emph{characteristic exponent} $p$ of a field $k$ is defined to be:
\begin{itemize}
    \item $p=\mathrm{char}(k)$ if $\mathrm{char}(k)>0$,
    \item $p=1$ if $\mathrm{char}(k)=0$.
\end{itemize}
\begin{lem}\label{lem:power-surjective}
    Let $\k$ be a commutative ring, $G$ be a reductive group over $\k$ and $U\subset G$ the unipotent radical of a Borel subgroup of $G$. Suppose $G$ acts rationally on a commutative $\k$-algebra $A$ by $k$-algebra automorphisms and $J\subset A$ is a $G$-stable ideal.
    \begin{enumerate}
        \item For any $x\in (A/J)^G$, there exists a positive integer $N$ such that $x^N$ lies in the image of $A^G\to (A/J)^G$. If moreover $\k$ is a field of characteristic exponent $p$, then $N$ can be chosen to be a power of $p$. 
        \item 
        (Similar statement for $U$-invariants) For any $x\in (A/J)^U$, there exists a positive integer $N$ such that $x^N$ lies in the image of $A^U\to (A/J)^U$. If moreover $\k$ is a field of characteristic exponent $p$, then $N$ can be chosen to be a power of $p$. 
    \end{enumerate}
\end{lem}
\begin{proof}
    (1) For a general ring $\k$, this follows from \cite[Proposition 4]{Kallen2014} and \cite[Theorem 12]{Kallen2010}. Next we assume that $\k=k$ is a field and prove the stronger statement. If $\mathrm{char}(k)=0$ this follows since taking $G$-invariants is an exact functor by complete reducibility. If $p=\mathrm{char}(k)>0$ the statement follows from the argument of \cite[proof of Proposition 41]{Kallen2010}, which we review for the convenience of the reader. Let $Y$ be an independent variable and let $G$ act on the polynomial ring $k[Y]$ trivially. By the case of general base ring, the map $A[Y]^G\to(A/J)[Y]^G$ is power-surjective. Take any element $x\in(A/J)^G$, then there exists $N>0$ such that $(x+Y)^N$ lifts to an element in $A[Y]^G=A^G[Y]$. Write $N=np^m$ where $n$ is prime to $p$ and $m\ge0$. Then we have $(x+Y)^N=(x^{p^m}+Y^{p^m})^n$ and the coefficient $nx^{p^m}$ of $Y^{(n-1)p^m}$ lifts to $A^G$. Since $n$ is invertible in $k$, we get that $x^{p^m}$ lifts to $A^G$.\par 
    (2) Let $\tilde{A}\coloneqq A\otimes_{\k}\k[G/U]$ and equip it with the diagonal $G$-action. Then $\tilde{J}\coloneqq J\otimes_\k\k[G/U]$ is a $G$-stable ideal of $\tilde{A}$. Since $\k[G/U]$ is a flat $\k$-algebra, we have a canonical isomorphism $\tilde{A}/\tilde{J}\simeq(A/J)\otimes_\k\k[G/U]$. After taking $G$-invariants and applying Lemma \ref{lem:transfer-principle} we get isomorphisms
    \[\tilde{A}^G\simeq A^U,\qquad (\tilde{A}/\tilde{J})^G\simeq((A/J)\otimes_\k\k[G/U])^G\simeq(A/J)^U.\]
    Then we conclude by (1).
\end{proof}

\subsection{Quantized coordinate ring of the basic affine space}
Fix a based root datum $\Psi$ and the associated split reductive group scheme $(G,B,T)$ over $\ZZ$ as in \S\ref{sec:notation-reductive-group}. Recall that the quotient $G/U$ is quasi-affine and its affine closure is the so-called \emph{basic affine space}:
\[\overline{G/U}^{\mathrm{aff}}\coloneqq\Spec(\ZZ[G]^U).\]
In this subsection we study the quantization of the basic affine space and prove some related technical results to be used later. Unlike the previous subsections, we now let $R$ be a general commutative $\A$-algebra (so the image of $v$ in $R$ is not necessarily $1$).\par
The following result describes the quantization of $\ZZ[G]^U$ in terms of the dual canonical basis, and justifies the notation in Definition \ref{def:G/U}.
\begin{prop}\label{prop:U-invariant}
    We have equalities of $R$-submodules of ${}_R\bfO$:
    \[H^0(\rho_r({}_R\hat{\bfU}^+),{}_R\bfO)={}_R\bfO_{G/U},\quad H^0(\rho_l({}_R\hat{\bfU}^-),{}_R\bfO)={}_R\bfO_{U^-\backslash G}.\]
\end{prop}
\begin{proof}
    We only need to prove the equality for ${}_R\hat{\bfU}^-$ (the other one will follow by applying the involution $\omega$). Let $\la\in X^+$ and $b\in\bfB(\la)$ so that $1_\la\sigma(b)^+\in\dot{\bfB}[\la]$. We claim that $(1_\la\sigma(b)^+)^*$ is left ${}_R\hat\bfU^-$-invariant. For this we may assume $R=\A$ and it suffices to show that for any $b'\in\bfB\backslash\{1\}$ and $x\in{}_{\zeta}\dot{\bfB}_{\chi}$ with $\zeta,\chi\in X$ (see Definition \ref{def:B-dot}), we have 
    \[(1_\la\sigma(b)^+)^*(b'^-1_{\zeta}x)=0.\] 
    Since $1_\la\sigma(b)^+\in 1_\la\dot{\bfU}1_{\la-\nu}$ where $\nu=|b|$, the identity above holds if $\zeta-|b'|\ne\la$. Now assume that $\zeta-|b'|=\la$. According to \cite[Proposition 25.2.6]{Lusztig-IntroQuantumGroup} we have $b'^-1_\zeta\in\dot{\bfB}$. Let $\mu\in X^+$ be the dominant weight such that $b'^-1_\zeta\in\dot{\bfB}[\mu]$. Then $b'^-1_\zeta$ does not act by zero on the Weyl module $\La_\mu$ and we have $\zeta\le\mu$. So we get inequalities
    \[\la\le\la+|b'|=\zeta\le\mu\]
    and since $b'\ne1$ we have strict inequality $\la<\mu$. Since $\dot{\bfU}[\ge\mu]$ is a two-sided ideal of $\dot{\bfU}$ we have $b'^-1_\zeta x\in\dot{\bfU}[\ge\mu]$ and hence $(1_\la\sigma(b)^+)^*(b'^-1_{\zeta}x)=0$. This proves the claim and consequently we verified the inclusion 
    \[{}_R\bfO_{U^-\backslash G}\subset H^0(\rho_l({}_R\hat{\bfU}^-),{}_R\bfO).\]
    Now take any $\rho_l({}_R\hat{\bfU}^-)$-invariant element $\varphi\in{}_R\bfO$. Let $\Supp(\varphi)=\{\la_1,\dotsc,\la_n\}$ be the smallest subset of $X^+$ such that $\varphi(c)=0$ for any $c\in\dot{\bfB}\setminus\bigcup_{i=1}^n\dot{\bfB}[\la_i]$. After reordering, we may and do assume that $\la_i\nleq\la_j$ for any $i>j$. 
    We will show by induction on $n$ that $\varphi\in{}_R\bfO_{U^-\backslash G}$. According to Proposition \ref{prop:parametrize-B[la]}, any element of $\dot{\bfB}[\la_n]$ can be written as $\beta_{\la_n}(b,b')=b^-1_{\la_n}\sigma(b')^++\alpha$ where $\alpha\in{}_R\dot{\bfU}[>\la_n]$ and $b,b'\in \bfB(\la_n)$. By assumption $\varphi(\alpha)=0$. Since $\varphi$ is left ${}_R\hat{\bfU}^-$-invariant, we have $\varphi(b^-1_{\la_n}\sigma(b')^+)=0$ when $b\ne1$. 
    Let 
    \[\varphi'\coloneqq\varphi-\sum_{b'\in\bfB(\la_n)}\varphi(1_{\la_n}\sigma(b')^+)(1_{\la_n}\sigma(b')^+)^*.\] 
    Then we have $\Supp(\varphi')=\Supp(\varphi)\setminus\{\la_n\}$
    and by what we have already shown $\varphi'$ is $\rho_l({}_R\hat{\bfU}^-)$-invariant. Therefore by induction hypothesis we get $\varphi'\in{}_R\bfO_{U^-\backslash G}$ and the same holds for $\varphi$.
\end{proof}

\begin{cor}\label{cor:U-inv-O(P)}
    Let $P\subset X^+$ be a subset that is downward closed. Then ${}_R\bfO(P)$ is an ${}_R\hat{\bfU}$-submodule of ${}_R\bfO$ under both left and right regular representations, and we have 
    \[H^0(\rho_r({}_R\hat{\bfU}^+)\times\rho_l({}_R\hat{\bfU}^-),{}_R\bfO(P))=\bigoplus_{\la\in P}R\cdot1_\la^*.\]
    Furthermore if $P$ is a submonoid of $X^+$, then this space is a subalgebra isomorphic to the monoid algebra $R[P]$.
\end{cor}
\begin{proof}
    By Proposition \ref{prop:downward-closed}, ${}_R\bfO(P)$ is a sub-coalgebra of ${}_R\bfO$. From this one deduces that $\mu_l=(S_\bfO\otimes\mathrm{Id})\circ\Delta_\bfO$ and $\mu_r={}^t\Delta_\bfO$ restrict to homomorphisms
    \[\mu_l,\mu_r\colon{}_R\bfO(P)\to{}_R\bfO\otimes_R{}_R\bfO(P)\]
    which means that ${}_R\bfO(P)$ is an ${}_R\hat{\bfU}$-submodule under both left and right regular representations. The description of invariants follows from Proposition \ref{prop:U-invariant}. The last statement follows from Lemma \ref{lem:filtration-O-multiplicative}.
\end{proof}
We end with the following technical result which will be used for proving the main theorem of \cite{CJ2}.
\begin{prop}\label{prop:nilpotent-deformation}
    Let $R$ be a commutative $\A$-algebra and let $I\subset R$ be a nilpotent ideal. Let $Q\subset P\subset X^+$ be two subsets such that ${}_R\bfO(Q)$ and ${}_R\bfO(P)$ are both $\rho_l({}_R\hat{\bfU})\times \rho_r({}_R\hat{\bfU})$-submodules of ${}_R\bfO$. Let $N\subset{}_R\bfO(P)$ be a $\rho_l({}_R\hat{\bfU})\times \rho_r({}_R\hat{\bfU})$-submodule satisfying the following conditions:
    \begin{enumerate}
        \item[(i)] ${}_R\bfO(P)/N$ is flat as $R$-module;
        \item[(ii)] $N\subset{}_R\bfO(Q)+I\cdot{}_R\bfO(P)$.
    \end{enumerate}
    Then we have $N\subset{}_R\bfO(Q)$. If moreover the following condition is also satisfied:
    \begin{enumerate}
        \item [(iii)] $N\otimes_R R/I={}_{R/I}\bfO(Q)$ as subsets of ${}_{R/I}\bfO(P)$,
    \end{enumerate}
    then we have $N={}_R\bfO(Q)$.
\end{prop}
\begin{proof}
    Since ${}_R\bfO(P)$ and ${}_R\bfO(Q)$ are preserved by both left and right regular actions of ${}_R\hat{\bfU}$, they are sub-coalgebras of ${}_R\bfO$. Therefore $P$ and $Q$ are both downward closed subsets of $X^+$ by Proposition \ref{prop:downward-closed}. For any $\la\in P\setminus Q$, the weight space ${}_R\bfO(Q)_{(-\la,\la)}=0$ since otherwise there would exist $\mu\in Q$ with $\dot{\bfB}[\mu]\cap 1_\la{}_R\dot{\bfU}1_\la\ne\varnothing$ and so $\la\le\mu$, contradicting the fact that $Q$ is downward closed. Hence by condition (ii) we have $N_{(-\la,\la)}\subset I\cdot{}_R\bfO(P)$. Since ${}_R\bfO(P)/N$ is $R$-flat, its direct summand $M\coloneqq{}_R\bfO(P)_{(-\la,\la)}/N_{(-\la,\la)}$ is also $R$-flat. In particular we have $\mathrm{Tor}_1^R(M,R/I)=0$. Let $\Omega\subset\dot{\bfB}^*$ be the subset that forms an $R$-basis for ${}_R\bfO(P)_{(-\la,\la)}$. Then the images of elements of $\Omega$ in ${}_{R/I}\bfO$ form a basis of the free $R/I$-module $M/IM={}_{R/I}\bfO(P)_{(-\la,\la)}$. From \cite[\href{https://stacks.math.columbia.edu/tag/051H}{Lemma 051H}]{Stacks}, we deduce that $M$ is a free $R$-module with basis $\Omega$ and so we must have $N_{(-\la,\la)}=0$.\par 
    Now take any $\la_1,\la_2\in X$ and suppose there exists an element $0\ne\varphi\in N_{(-\la_1,\la_2)}\setminus{}_R\bfO(Q)$. Let $\mu\in P\setminus Q$ be a maximal element such that there exists $b\in\dot{\bfB}[\mu]\cap1_{\la_1}{}_R\dot{\bfU}1_{\la_2}$ with $\varphi(b)\ne0$. According to Proposition \ref{prop:parametrize-B[la]} we can write $b=\beta_\mu(b_1,b_2)$ with $b_1,b_2\in\bfB(\mu)$ and $\la_1+|b_1|=\la_2+|b_2|=\mu$. Let $\tilde{\varphi}\coloneqq\rho_l(S^{-1}(b_1^-1_{\mu}))\rho_r(1_{\mu}\sigma(b_2)^+)\cdot\varphi$. Then $\tilde\varphi\in N_{(-\mu,\mu)}$ and $\tilde{\varphi}(c)=\varphi(b_1^-1_{\mu}c1_{\mu}\sigma(b_2)^+)$ for all $c\in\dot{\bfB}$. Since $b-b_1^-1_\mu\sigma(b_2)^+\in{}_R\dot{\bfU}[>\mu]$, we have $\varphi(b-b_1^-1_\mu\sigma(b_2)^+)=0$. (Indeed, Suppose there exists $\nu>\mu$ with and element $b'\in\dot{\bfB}[\nu]$ such that $\varphi(b')\ne0$. Since $\varphi\in N\subset{}_R\bfO(P)$ we have $\nu\in P$ and if $\nu\in Q$ then $\mu\in Q$ by the downward-closedness of $Q$, contradicting the assumption that $\mu\notin Q$. So we would have $\nu\in P\setminus Q$ but this contradicts with the maximality of $\mu$.) Thus we get that $\tilde{\varphi}(1_\mu)\ne0$. On the other hand we have $N_{(-\mu,\mu)}=0$ by the previous paragraph. This is a contradiction and we conclude that $N_{(-\la_1,\la_2)}\subset{}_R\bfO(Q)_{(-\la_1,\la_2)}$ for all $\la_1,\la_2\in X$. Then we deduce that $N\subset{}_R\bfO(Q)$.\par 
    Finally, suppose that (iii) also holds. In the short exact sequence of $R$-modules
    \[0\to{}_R\bfO(Q)/N\to{}_R\bfO(P)/N\to{}_R\bfO(P)/{}_R\bfO(Q)\to0\]
    the last term is clearly $R$-flat, and the middle term is $R$-flat by (i). So we deduce that ${}_R\bfO(Q)/N$ is also $R$-flat. Thus  condition (iii) and the short exact sequence 
    \[0\to N\to{}_R\bfO(Q)\to {}_R\bfO(Q)/N\to0\]
    imply that $({}_R\bfO(Q)/N)\otimes_RR/I=0$ and hence ${}_R\bfO(Q)=N+I\cdot{}_R\bfO(Q)$. We iterate this identity to get
    \begin{align*}
        {}_R\bfO(Q)=N+I\cdot{}_R\bfO(Q)&=N+I\cdot(N+I\cdot{}_R\bfO(Q))\\ &=N+I^2\cdot{}_R\bfO(Q)=\dotsm=N+I^m\cdot{}_R\bfO(Q)
    \end{align*}
    which holds for any integer $m\ge1$. Since $I$ is a nilpotent ideal, we conclude that $N={}_R\bfO(Q)$.
\end{proof}

\end{document}